\documentclass[a4 paper,11pt]{article}
\usepackage{stmaryrd}
\usepackage{amsfonts}
\usepackage{bbm}
\usepackage{amscd}
\usepackage{mathrsfs}
\usepackage{latexsym,amssymb,amsmath,amscd,amscd,amsthm,amsxtra,xypic}
\usepackage[dvips]{graphicx}
\usepackage[utf8]{inputenc}
\usepackage[T1]{fontenc}
\usepackage{enumerate}
\usepackage{amssymb}
\usepackage[all]{xy}
\usepackage{ifpdf}
\ifpdf
\usepackage[colorlinks=true,linkcolor=blue,final,backref=page,hyperindex,citecolor=red]{hyperref}
\else
\usepackage[colorlinks,final,backref=page,hyperindex,hypertex]{hyperref}
\fi

\usepackage{nicefrac,mathtools,enumitem}
\usepackage{microtype}
\usepackage{amsfonts}
\usepackage{amssymb}
\usepackage{mathrsfs}
\usepackage{amsmath}
\usepackage{amsthm}
\usepackage{enumerate}
\usepackage{indentfirst}
\usepackage{amsfonts}
\usepackage{amssymb}
\usepackage{mathrsfs}
\usepackage{amsmath}
\usepackage{amsthm}
\usepackage{enumerate}
\usepackage{cite}
\usepackage{mathrsfs}
\usepackage{geometry}
\usepackage{multirow}
\allowdisplaybreaks

\usepackage{dsfont}

\def\<{\langle}
\def\>{\rangle}

\def\s{\sigma}

\newcommand{\bmx}{\begin{pmatrix}}
\newcommand{\emx}{\end{pmatrix}}

\newtheorem{thm}{Theorem}[section]
\newtheorem{lem}[thm]{Lemma}
\newtheorem{cor}[thm]{Corollary}
\newtheorem{pro}[thm]{Proposition}
\newtheorem{ex}[thm]{Example}

\theoremstyle{definition}
\newtheorem{defi}{Definition}[section]

\theoremstyle{remark}
\newtheorem{rmk}{Remark}[section]
\begin{document}
\title{\sf Extensions of pseudo-euclidean mock-Lie superalgebras}
\date{}
\author{ Tahar Benyoussef$^{1}$\footnote{   E-mail: tahar.benyoussef40@gmail.com}~~ and Sami Mabrouk$^{2}$\footnote{   E-mail: mabrouksami00@yahoo.fr, sami.mabrouk@fsgf.u-gafsa.tn}}
\date{\small{1.  Faculty of Sciences, University of Sfax,   BP
1171, 3000 Sfax, Tunisia.
 }\\
 \small{2. University of Gafsa, Faculty of Sciences Gafsa, 2112 Gafsa, Tunisia.
 }}

\maketitle

\begin{abstract}
In this article, we introduce mock-Lie superalgebras, we give some definitions, properties, constructions, and we study their representations. Moreover we introduce pseudo-euclidean mock-Lie superalgebras which are mock-Lie superalgebras with even non-degenerate supersymmetric and invariant bilinear forms. Finally, we study the double extensions and generalized double extensions of mock-Lie superalgebras and their isometries.
\end{abstract}



\numberwithin{equation}{section}

\tableofcontents

%%%%%%%%%%%%%%%%%%%%%%%%%%%%%%%%%%%%%%%%%%%
\section{Introduction}

The class of mock-Lie algebras represents a relatively new category that has emerged in mathematical literature as an example of Jordan algebras. These algebras are also referred to as Jacobi-Jordan algebras. They are characterized as commutative algebras that satisfy the Jacobi identity. First introduced in \cite{B.F}
mock-Lie algebras have since garnered significant attention, with numerous studies conducted on the topic, as evidenced by works such as  \cite{P.Z,KAST1,KAST2,C.G}. Often regarded as the "strange cousins" of Lie algebras, these algebras have been explored in greater depth in various references, including \cite{Attan,Agore1,Agore2,Haliya,C.G, Benali, MaMabroukMakhloufSong, Hou2}, which provide further insights and investigations into their properties and applications. Following this relation with Lie algebras and Jordan algebras, it is normal to base first on the notions of Lie superalgebras to take the first step in the superization of mock-Lie superalgebra and secondly on Jordan superalgebras based on a proposition in this article which shows that any mock-Lie superalgebra is a Jordan superalgebra ( see Proposition \ref{Jacob-Jor.Is.Jor}). In this sense it is convenient to  talk in this introduction about the history of quadratic Lie and quadratic Jordan superalgebras and the notion of double extensions of these superalgebras. 

Quadratic Lie algebras constitute a fundamental object of investigation in contemporary mathematical physics due to their intrinsic algebraic properties and broad theoretical implications. These algebras are formally characterized by the existence of a non-degenerate, Ad-invariant symmetric bilinear form, which establishes their significance across multiple mathematical disciplines. Geometrically, they correspond precisely to Lie algebras of Lie groups equipped with bi-invariant pseudo-Riemannian metric tensors, with the algebraic form representing the metric's restriction to the tangent space at the identity. This correspondence facilitates their application in the analysis of symmetric spaces, homogeneous pseudo-Riemannian manifolds, and exceptional holonomy groups. Within quantum field theory, these algebras assume particular importance through their role in the Sugawara construction \cite{JMF}, which provides an isomorphism between the Virasoro algebra and quadratic subalgebras of affine Lie algebras. This construction fundamentally relates the stress-energy tensor to conserved currents in two-dimensional conformal field theories, with applications extending to Wess-Zumino-Novikov-Witten models and current algebra representations. The construction's feasibility is contingent upon the existence of the invariant form, illustrating the profound physical consequences of this algebraic property.

The classification theory for these structures has been revolutionized by the development of double extension methodologies. The seminal work of Medina and Revoy \cite{Md} introduced this approach through the systematic combination of central extensions and semidirect product operations, yielding an inductive classification scheme for quadratic Lie algebras. This framework provides both constructive techniques and deep structural characterization of these algebraic objects. Subsequent theoretical advancements have substantially generalized this framework. Benamor and Benayadi \cite{BS} established its extension to $\mathbb{Z}_2-$ graded Lie superalgebras, enabling classification of quadratic Lie superalgebras with applications in supersymmetric field theories. Bajo et al. \cite{Bajo} subsequently incorporated solvability conditions, producing a complete classification of solvable metric Lie superalgebras. $\Gamma$-graded Lie algebras with quadratic-algebraic structures, that is  $\Gamma$-graded Lie algebras provided with homogeneous, symmetric, invariant and nondegenerate bilinear forms,  have been extensively studied in \cite{AmmarAyadiMabroukMakhlouf}.
 In \cite{Bak}, the authors focuses on the most general case of solvable pseudo-Euclidean Jordan superalgebras. These developments collectively demonstrate the method's capacity for unified treatment of diverse algebraic structures while maintaining rigorous inductive foundations. The classical case of a symplectic or pseudo-euclidean or symplectic and pseudo-euclidean mock-Lie algebra has been recently studied by Baklouti et al \cite{BakBen1SymJ.J.A}. The continued refinement of this theoretical framework suggests potential applications in higher-dimensional conformal field theories, exceptional geometry, and the classification of non-associative algebraic structures.

The structure of the article is as follows: In Section \ref{section.2}, we introduce mock-Lie superalgebras and show that every mock-Lie superalgebra is a Jordan superalgebra. We explore the representation theory of mock-Lie superalgebras, providing key examples known as the adjoint and dual representations. Additionally, we offer a characterization through
the semi-direct product. In Section \ref{section.3}, we introduce the notion of mock-Lie superalgebra with an even supersymmetric, invariant and non-degenerate bilinear form called pseudo-euclidean  mock-Lie superalgebra and give some construction and characterization. In Section \ref{section.4}, given a pseudo-euclidean mock-Lie superalgebra $\mathcal{J}_1$ and a mock-Lie super algebra $\mathcal{J}_2$, we construct a mock-Lie superalgebra $\mathcal{J}_1\oplus{\mathcal{J}_2}^*$ which is called central extension of $\mathcal{J}_1$ by ${\mathcal{J}_2}.$  Furthermore we show that $\mathcal{J}_2\oplus\mathcal{J}_1\oplus{\mathcal{J}_2}^*$ is a mock-Lie superalgebra which is called double extension of $\mathcal{J}_1$ by $\mathcal{J}_2,$ we construct in addition a new supersymmetric invariant bilinear form on $\mathcal{J}_2\oplus\mathcal{J}_1\oplus{\mathcal{J}_2}^*.$ Finally, we introduce the notion of generalized double extensions of mock-Lie superalgebra by the one dimentional odd mock-Lie superalgebra $\mathbb{K}u$ with $u\in\mathcal{J}_{\bar{1}}.$ Moreover we investigate the notion of isometry betwean tow generalized double extensions of pseudo-euclidean mock-Lie superalgebra.

Throughout this paper, all algebras and vector spaces are finite-dimensional and over a field
$\mathbb{K}$ of characteristic $0$.
%%%%%%%%%%%%%%%%%%%%%%%%%%%%%%%%%%%%%%%%%%%
\section{Mock-Lie superalgebras}\label{section.2}

In the following, we study  mock-Lie superalgebras and discuss some of their properties. The representation theory of mock-Lie superalgebras is given.
Let $\mathcal{V}=\mathcal{V}_{\bar{0}}\oplus \mathcal{V}_{\bar{1}}$ be a $\mathbb{Z}_{2}$-graded vector space over the field $\mathbb{K}$ of characteristic $0$. An element $x$ in $\mathcal{V}$ is called homogeneous if $x\in \mathcal{V}_{\bar{0}}$ or $x\in \mathcal{V}_{\bar{1}}$. Throughout this paper, all elements are supposed to be homogeneous unless otherwise stated. For a homogeneous element $x$ we shall use the standard notation $\vert x\vert\in \mathbb{Z}_{2}=\{\bar{0},\bar{1}\}$ to indicate its degree, i.e. whether it is contained in the even part ($\vert x\vert=\bar{0}$) or in the odd part ($\vert x\vert=\bar{1}$). All  superalgebras considered in this paper are finite-dimensional.
Let $\mathcal{V}$ be a $\mathbb{K}-$vector space and $\gamma:= {\otimes \mathcal{V}}$ a Grassmann (or exterior) algebra of $\mathcal{V}$. We know that $\Gamma:= \oplus_{i\in \mathbb{Z}}\otimes^i\mathcal{V}= \Gamma_{\bar{0}}\oplus \Gamma_{\bar{1}}$ is a
$\mathbb{Z}_{2}-$graded associative algebra, where $ \Gamma_{\bar{0}}:= \oplus_{i\in \mathbb{Z}}\otimes^{2i}\mathcal{V}$ and $ \Gamma_{\bar{1}}:= \oplus_{i\in \mathbb{Z}}\otimes^{2i+1}\mathcal{V}$, such that
$X_\alpha X_\beta= (-1)^{\alpha\beta }X_\beta X_\alpha, \,\, \forall \, (X_\alpha,X_\beta)\in \Gamma_\alpha\times \Gamma_\beta.$
\subsection{ Definition and constructions}
Let $(\mathcal{A}=\mathcal{A}_{\bar{0}}\oplus \mathcal{A}_{\bar{1}}, \cdot)$ be a superalgebra and $\Gamma(\mathcal{A})$ its Grassmann enveloping algebra which is a subalgebra of $\mathcal{A}\otimes \Gamma$ given by $\Gamma(\mathcal{A})=\mathcal{A}_{\bar{0}}\otimes \Gamma_{\bar{0}}\oplus \mathcal{A}_{\bar{1}}\otimes \Gamma_{\bar{1}}$. Let us assume now that $\mathcal{M}$ is a homogeneous variety of algebras. Then, $\mathcal{A}$ is said to be a $\mathcal{M}$-superalgebra if $\Gamma(\mathcal{A})$ belongs to $\mathcal{M}$ (see \cite{ZS}). So, $\mathcal{A}=\mathcal{A}_{\bar{0}}\oplus \mathcal{A}_{\bar{1}}$ is a mock-Lie superalgebra if $\Gamma(\mathcal{A})$ is a mock-Lie algebra. Consequently, we get the following definition.
\begin{defi}\label{d11}
Let $\mathcal{J}=\mathcal{J}_{\bar{0}}\oplus \mathcal{J}_{\bar{1}}$ be a $\mathbb{Z}_{2}$-graded vector space and let $``\bullet": \mathcal{J} \otimes \mathcal{J} \rightarrow \mathcal{J}$ be an even  bilinear map on $\mathcal{J}$ (i.e $\mathcal{J}_{i} \bullet \mathcal{J}_{j}\subset\mathcal{J}_{i+j},~\forall i,j\in \mathbb{Z}_{2}$). Then
$(\mathcal{J},\bullet)$ is called a mock-Lie superalgebra if, for all $x,y,z\in \mathcal{J}_{\vert x\vert}\times \mathcal{J}_{\vert y\vert}\times \mathcal{J}_{\vert z\vert}$, the following equations are satisfied:
{\small\begin{align}
& x\bullet y=(-1)^{\vert x\vert \vert y\vert} y\bullet x,\quad \text{ (Super-commutativity)}, \label{S.Commt.Ident}\\
& (-1)^{\vert x\vert \vert z\vert}x\bullet(y\bullet z)+
(-1)^{\vert x\vert \vert y\vert}y\bullet(z\bullet x)+(-1)^{\vert y\vert \vert z\vert}z\bullet(x\bullet y)=0,\quad \text{ (super-Jacobi~identity)}. \label{S.Jacob.Ident}
\end{align}}
The super-Jacobi identity \eqref{S.Jacob.Ident} is equivalent to $$x\bullet(y\bullet z)=-(x\bullet y)\bullet z-(-1)^{\vert x\vert \vert y\vert} y\bullet (x\bullet z),$$ meaning that $L_{x}$ is an anti-superderivation of $\mathcal{J}$ of degree $\vert x\vert$, where $L_{x}$ is the left multiplication by $x$ defined by $L_{x}(y):=x\bullet y$.
Recall that an endomorphism $D\in{\rm End}(\mathcal{J})_{\alpha}$ is said to be:
\begin{enumerate}
\item[\bf (1)] superderivation of $\mathcal{J}$ of degree $\alpha$ if, for all $x\in \mathcal{J}_{\vert x\vert},\, y\in \mathcal{J},$~~$$D(x \bullet y)=D(x)\bullet y+(-1)^{\alpha \vert x\vert}x\bullet D(y),$$
\item[\bf (2)] anti-superderivation of $\mathcal{J}$ of degree $\alpha$ if, for all $x\in \mathcal{J}_{\vert x\vert},\, y\in \mathcal{J}$,~~$$D(x \bullet y)=-D(x)\bullet y-(-1)^{\alpha \vert x\vert}x\bullet D(y).$$
\end{enumerate}
\end{defi}
The set of all superderivations and anti-superderivations of $\mathcal{J}$ will be denoted by $\mathfrak{Der}(\mathcal{J})  .$ and $\mathfrak{AnDer}(\mathcal{J})$ respectively. Clearly that $\mathfrak{Der}(\mathcal{J})$ and $\mathfrak{AnDer}(\mathcal{J})$ are  graded superspaces. 
%Let $\mathfrak{Der}(\mathcal{J})$ (respectively $\mathfrak{AnDer}(\mathcal{J})$ the subset of $\mathfrak{Der}(\mathcal{J})$ of all superderivations ( respectively all anti-superderivations) of $\mathcal{J}$. $\mathfrak{Der}(\mathcal{J})$ (respectively $\mathfrak{AnDer}(\mathcal{J})$ is a graded superspace.
%%% connexion with Jordan superalgebra%%
   A homomorphism of  mock-Lie superalgebras $(\mathcal{J}_1,\bullet_1)$ to $(\mathcal{J}_2,\bullet_2)$ is an even linear map  $\Psi:(\mathcal{J}_1,\bullet_1)\to (\mathcal{J}_2,\bullet_2)$, such that $\Psi(x\bullet_1y)=\Psi(x)\bullet_2 \Psi(y),$ for $ x,y\in\mathcal{J}_1.$
   
In the following, we recall the notation of Jordan superalgebras introduced in \cite{MartínezZelmanov} (see also \cite{Abdaoui, MabroukNcib}), and we show the connexion between the Jordan superalgebra and the mock-Lie superalgebra. 
\begin{defi}
    A superalgebra $\mathcal{J}$ over a field $\mathbb{K}$ is called a Jordan superalgebra if it equipped by an even super-commutative product $\bullet$ satisfies the following two identity
    \begin{align}\label{SuperJorIdentity}
       \nonumber &(-1)^{|a||c|}(a\bullet b)(c\bullet d)+(-1)^{|a||b|}(b\bullet c)(a\bullet d)+(-1)^{|c||b|}(c\bullet a)(b\bullet d)\\=&(-1)^{|a||c|}a\bullet((b\bullet c)\bullet d)+(-1)^{|a||b|}b\bullet((c\bullet a)\bullet d)+(-1)^{|c||b|}c\bullet((a\bullet b)\bullet d)
        \end{align}
    for all $a\in \mathcal{J}_{\vert a\vert},b\in \mathcal{J}_{\vert b\vert},c\in {J}_{\vert c\vert}\quad and\quad d \in \mathcal{J}$.
\end{defi}
With $a=b=c=x,d=y$, the identity \eqref{SuperJorIdentity} can be written
\begin{align*}
   & (-1)^{|x||x|}x^2 \bullet(x\bullet y)+(-1)^{|x||x|}x^2 \bullet(x\bullet y)+(-1)^{|x||x|}x^2 \bullet(x\bullet y)\\=&(-1)^{|x||x|}x\bullet(x^2 \bullet y)+(-1)^{|x||x|}x\bullet(x^2 \bullet y)+(-1)^{|x||x|}x\bullet(x^2 \bullet y).
   \end{align*}
   Then we have $x^2\bullet(x\bullet y)=x\bullet(x^2 \bullet y).$
   
   Now using the super-commutativity we have  
   $(-1)^{|x||y|}x^2 \bullet(y\bullet x)=(-1)^{|x|(|x|+|x|+|y|)}(x^2\bullet y)\bullet x.$ Finaly we have the identity 
   \begin{equation}\label{new-jordan-id}
       x^2 \bullet(y\bullet x)=(x^2\bullet y)\bullet x 
   \end{equation}
   for all $x,y\in \mathcal{J}_{\vert x\vert}\times \mathcal{J}_{\vert y\vert}.$ 
   Hence $\mathcal{J}$ is a Jordan superalgebra if $\mathcal{J}$ a super-commutative and the identity \eqref{new-jordan-id} holds.
   \begin{pro}\label{Jacob-Jor.Is.Jor}
       Every mock-Lie superalgebra is a Jordan superalgebra verifying $x^3=0$.
   \end{pro}
   \begin{proof}
       By the super-Jacobi identity and the super-commutatuvity we have $x^3=0.$ According to  the super-Jacobi identity with $z=x$ we have, for all $x,y\in \mathcal{J}_{\vert x\vert}\times \mathcal{J}_{\vert y\vert}$, 
       $$(-1)^{|x||x|}x\bullet(y\bullet x)+(-1)^{|x||y|}y\bullet x^2 +(-1)^{|x||y|}x\bullet(x\bullet y)=0.$$
       This identity equivalent to 
       $$(-1)^{|x||y|}x^2\bullet y= -(1+(-1)^{|x||x|})x\bullet(y\bullet x).$$ 
       Now replacing $y$ by $y\bullet x$ yields 
       \begin{align*}
       (-1)^{|x||y|}x^2\bullet(y\bullet x)=&-(1+(-1)^{|x||x|})x\bullet((y\bullet x)\bullet x)\\=&x\bullet\big(-(1+(-1)^{|x||x|})(y\bullet x)\bullet x\big)\\=&(-1)^{|x||x|+|x||y|}x\bullet \big(-(1+(-1)^{|x||x|})x\bullet(y\bullet x)\big)\\=&(-1)^{|x||x|+|x||y|}x\bullet\big((-1)^{|x||y|}x^2\bullet y\big)\\=&x\bullet\big((-1)^{|x||x|}x^2 \bullet y\big)\\=&(-1)^{|x||y|}(x^2\bullet y)\bullet x.
       \end{align*}
       Hence $\mathcal{J}$ is a Jordan superalgebra.
       
       \end{proof}
       \begin{pro}
           If $(\mathcal{J},\bullet)$ is a mock-Lie superalgebra. Then,  for all $x,y\in \mathcal{J}_{|x|}\times\mathcal{J}_{|y|},$$$x^2\bullet(y\bullet x)=(x^2\bullet y)\bullet x=0.$$
          
       \end{pro}
       \begin{proof}
           By the super-Jacobi identity we have 
           $x^2\bullet (y\bullet x) +y\bullet x^3+(-1)^{|x||y|}x\bullet (x^2\bullet y)=0$. 
           Then $x^2\bullet (y\bullet x)=-(-1)^{|x||y|}x\bullet(x^2\bullet y)=-(x^2 \bullet y)\bullet x.$ 
           By \eqref{new-jordan-id} we have $(x^2\bullet y)\bullet x=x^2 \bullet(y\bullet x).$ Hence $x^2\bullet (y\bullet x)=(x^2 \bullet y)\bullet x=0,$ for all $x,y\in \mathcal{J}_{|x|}\times\mathcal{J}_{|y|}.$ \end{proof}
           In the following, we construct two new mock-Lie superalgebras: first, from two given superalgebras, and then from a mock-Lie superalgebra and an associative commutative superalgebra. The following two propositions have straightforward proofs.
           \begin{pro}
             Let $(\mathcal{J}_1=\{\mathcal{J}_1\}_{\bar{0}}\oplus \{\mathcal{J}_1\}_{\bar{1}},\bullet_1)$ and $(\mathcal{J}_2=\{\mathcal{J}_2\}_{\bar{0}}\oplus \{\mathcal{J}_2\}_{\bar{1}},\bullet_2)$ be mock-Lie superalgebras. Then $\mathcal{J}=\mathcal{J}_1\oplus\mathcal{J}_2=\mathcal{J}_{\bar{0}}\oplus\mathcal{J}_{\bar{1}}$ with $\mathcal{J}_{\bar{0}} =\{\mathcal{J}_1\}_{\bar{0}}\oplus \{\mathcal{J}_2\}_{\bar{0}}$ and $\mathcal{J}_{\bar{1}} =\{\mathcal{J}_1\}_{\bar{1}}\oplus \{\mathcal{J}_2\}_{\bar{1}},$ is a mock-Lie superalgebra with the product: $$(x_1\oplus y_1)\bullet(x_2\oplus y_2)=x_1\bullet_1 x_2\oplus y_1\bullet_2 y_2,\quad\forall x_1,x_2\in \mathcal{J}_1,y_1,y_2\in \mathcal{J}_2.$$
             \end{pro}
             \begin{pro}
                 Let $(\mathcal{J}=\mathcal{J}_{\bar{0}}\oplus \mathcal{J}_{\bar{1}},\bullet_1)$ be a mock-Lie superalgebra and $(\mathcal{A}=\mathcal{A}_{\bar{0}}\oplus\mathcal{A}_{\bar{1}},\cdot)$ be an associative commutative superalgebra. Then $\mathfrak{J}=\mathcal{J}\otimes\mathcal{A}=\mathfrak{J}_{\bar{0}}\oplus\mathfrak{J}_{\bar{1}}$ with $\mathfrak{J}_{\bar{0}} =\mathcal{J}_{\bar{0}}\otimes\mathcal{A}_{\bar{0}}\oplus\mathcal{J}_{\bar{1}}\otimes\mathcal{A}_{\bar{1}}$ and $\mathfrak{J}_{\bar{1}} =\mathcal{J}_{\bar{0}}\otimes\mathcal{A}_{\bar{1}}\oplus\mathcal{J}_{\bar{1}}\otimes\mathcal{A}_{\bar{0}},$ is a mock-Lie superalgebra with the product :
                 $$(x\otimes a)\bullet(y\otimes b)=x\bullet_1 y\otimes a\cdot b,\quad\forall x,y\in \mathcal{J},a,b\in \mathcal{A}.$$
                 \end{pro}
      % In the following we constract an example of mock-Lie superalgebra.
      % \begin{ex}\label{exp2.3}
  %\end{ex}
           %\begin{ex}
           %Let $\mathcal{J}=\mathcal{J}_{\bar{0}}\oplus\mathcal{J}_{\bar{1}}$ be a superalgebra over a field $\mathbb{K}$ of characteristic 0, defined as follows:

           %\item [1-]\textbf{Graded Structure:} 
           %\begin{itemize}
               %\item [a-] The even part is spanned by two elements $e_1,e_2.$
               %\item[b-] The odd part is spanned by two elements $f_1$ and $f_2.$
           %\end{itemize}
           %\item[2-] \textbf{Product:} Define the product "$\bullet$" on homogeneous elements as follows \\
           %\textbf{even-even}
          % $$e_1\bullet e_1=e_2\bullet e_2=e_1\bullet e_2=0.$$
           %\textbf{odd-odd}
          % $$f_1\bullet f_1=f_2\bullet f_2=0,f_1\bullet f_2=0.$$
           %\textbf{even-odd}
           %$$e_1\bullet f_1=\frac{1}{2}f_1, e_1\bullet f_2=\frac{1}{2}f_2,e_2\bullet f_1=\frac{1}{2}f_1, e_2\bullet f_2=-\frac{1}{2}f_2.$$
           %Extend the product to all elements by linearity and supercommutativity: $x\bullet y=(-1)^{|x||y|}y\bullet x.$
          
           %\end{ex}
\subsection{Representation of mock-Lie superalgebras}
In the following, we explore the representation theory of mock-Lie superalgebras, providing key examples known as the adjoint and dual representations. Additionally, we offer a characterization through the semi-direct product.
\begin{defi}
A \textbf{representation} of a mock-Lie  superalgebra $(\mathcal{J},\bullet)$ is a pair $(\mathcal{V},\pi)$ where $\mathcal{V}$ is a $\mathbb{Z}_{2}$-graded vector space and $\pi:\mathcal{J}\rightarrow End(\mathcal{V})$ is an even linear map such that the following condition holds, for all $x\in \mathcal{J}_{\vert x\vert},\, y\in \mathcal{J}_{\vert y\vert}$,
\begin{equation}\label{RepCond}
    \pi(x\bullet y)=-\pi(x)\pi(y)-(-1)^{ \vert x\vert \vert y\vert}\pi(y)\pi(x).
\end{equation}
\end{defi}
\begin{ex} \label{adjrep}\rm
Let $(\mathcal{J},\bullet)$ be a mock-Lie superalgebra. Then the even linear map $ad:\mathcal{J}\rightarrow End(\mathcal{J})$ given by: $ad(x)(y)=x\bullet y,~\forall x,y\in \mathcal{J}$, is a representation of $\mathcal{J}$, which is called the adjoint representation of $\mathcal{J}$.
\end{ex}
\begin{pro} \label{p18}
Let $(\mathcal{J},\bullet)$ be a mock-Lie superalgebra, $\mathcal{V}$ be a $\mathbb{Z}_{2}$-graded vector space and $\pi:\mathcal{J}\rightarrow End(\mathcal{V})$ be an even linear map. Then, the $\mathbb{Z}_{2}$-graded vector space $\overline{\mathcal{J}}=\mathcal{J}\oplus \mathcal{V}$ endowed with the product defined by: $$(x+u) \star (y+v)=x\bullet y+\pi(x)(v)+(-1)^{\vert x\vert \vert y\vert}\pi(y)(u),~\forall (x+u,y+v)\in \overline{\mathcal{J}}_{\vert x\vert}\times \overline{\mathcal{J}}_{\vert y\vert}$$ is a mock-Lie superalgebra if and only if $(\mathcal{V},\pi)$ is a representation of $\mathcal{J}$.
\end{pro} 

\begin{proof}
Let  $x,y,z \in \mathcal{J}_{\vert x\vert}\times \mathcal{J}_{\vert y\vert}\times \mathcal{J}_{\vert z\vert}$ and  $u,v,w \in \mathcal{V}$ we have \begin{align*}
  &\circlearrowleft_{(x+u),(y+v),(z+w)}  (-1)^{|x||z|}(x+u)\star\big((y+v)\star(z+w)\big)\\=&(-1)^{|x||z|}(x+u)\star\big(y\bullet z+\pi(y)w+(-1)^{|z||y|}\pi(z)v\big)+
(-1)^{|y||x|}(y+v)\star\big(z\bullet x\\&+\pi(z)u+(-1)^{|z||x|}\pi(x)w\big)+(-1)^{|z||y|}(z+w)\star\big(x\bullet y+\pi(x)v+(-1)^{|y||x|}\pi(y)u\big)
\\=&(-1)^{|x||z|}\big(x\bullet (y\bullet z)+\pi(x)\pi(y)w+(-1)^{|z||y|}\pi(x)\pi(z)v+(-1)^{(|y|+|z|)|x|}\pi(y\bullet z)u\big)\\&
+(-1)^{|y||x|}\big(y\bullet (z\bullet x)+\pi(y)\pi(z)u+(-1)^{|z||x|}\pi(y)\pi(x)w+(-1)^{(|z|+|x|)|y|}\pi(z\bullet x)v\big)
\\&+(-1)^{|z||y|}\big(z\bullet (x\bullet y)+\pi(z)\pi(x)v+(-1)^{|y||x|}\pi(z)\pi(y)u+(-1)^{(|x|+|y|)|z|}\pi(x\bullet y)w\big)
\\=&(-1)^{|x||z|}x\bullet (y\bullet z)+(-1)^{|y||x|}y\bullet (z\bullet x)+(-1)^{|z||y|}z\bullet (x\bullet y)
+\big((-1)^{|x||z|}\pi(x)\pi(y)\\&+(-1)^{|y||x|+|z||x|}\pi(y)\pi(x)+(-1)^{ |z||x|}\pi(x\bullet y)\big)w
+\big((-1)^{|z||y|}\pi(z)\pi(x)+(-1)^{|x||z|+|y||z|}\pi(x)\pi(z)\\&+(-1)^{ |x||z|}\pi(z\bullet x)\big)v+\big((-1)^{|y||x|}\pi(y)\pi(z)+(-1)^{|y||z|+|x||y|}\pi(z)\pi(y)+(-1)^{ |x||y|}\pi(y\bullet z)\big)u\\
=&\big((-1)^{|x||z|}\pi(x)\pi(y)+(-1)^{|y||x|+|z||x|}\pi(y)\pi(x)+(-1)^{ |z||x|}\pi(x\bullet y)\big)w
+\big((-1)^{|z||y|}\pi(z)\pi(x)\\&+(-1)^{|x||z|+|y||z|}\pi(x)\pi(z)+(-1)^{ |x||z|}\pi(z\bullet x)\big)v
+\big((-1)^{|y||x|}\pi(y)\pi(z)+(-1)^{|y||z|+|x||y|}\pi(z)\pi(y)\\&+(-1)^{ |x||y|}\pi(y\bullet z)\big)u=0.\end{align*}
Then $(\overline{\mathcal{J}},\star)$ is a mock-Lie superalgebra if and only if $(\mathcal{V},\pi)$ is a representation of $\mathcal{J}$.
\end{proof}
\begin{defi}
Let $(\mathcal{J},\bullet)$ be a mock-Lie superalgebra and $(\mathcal{V},\pi)$ and $(\mathcal{V}',\pi')$ be two representations of $\mathcal{J}$. An even linear map $\Phi:\mathcal{V} \rightarrow \mathcal{V}'$ is called a morphism of representations if
\begin{equation}
    \pi'(x)\circ \Phi=\Phi \circ \pi(x),\;\forall x\in \mathcal{J}.
\end{equation}
If moreover $\Phi$ is bijective, then we say that $(\mathcal{V},\pi)$ and $(\mathcal{V}',\pi')$ are equivalent (or isomorphic).
\end{defi}
\begin{pro} \label{corep}
Let $(\mathcal{J}, \bullet)$ be a mock-Lie superalgebra and $(\mathcal{V},\pi)$ be a representation of $\mathcal{J}$. Let us consider the even linear map $\pi^*:\mathcal{J}\rightarrow End(\mathcal{V}^*)$ defined by:\[\pi^*(x)(f)=(-1)^{\vert f\vert \vert x\vert}f\circ \pi(x),~~\forall x\in \mathcal{J}_{\vert x\vert},~f\in \mathcal{J}^*_{\vert f\vert}~.\] Then, $(\mathcal{V}^*,\pi^*)$ is a representation of $\mathcal{J}$ on $\mathcal{V}^{*}$ which is called dual representation.
\end{pro}
\begin{proof}
For all $ x,y \in \mathcal{J}_{\vert x\vert}\times \mathcal{J}_{\vert y\vert}$, and for all $f \in \mathcal{V}^*$, we get
\begin{align*}
&\big(-\pi^*(x)\pi^*(y)-(-1)^{|x||y|}\pi^*(y)\pi^*(x)\big)\circ f\\
=&-\pi^*(x)\pi^*(y)\circ f-(-1)^{|x||y|}\pi^*(y)\pi^*(x)\circ f\\
=&(-1)^{|f||y|}\pi^*(x)\circ f \circ \pi(y)-(-1)^{|x||y|}(-1)^{|f||x|}\pi^*(y)\circ f\circ \pi(x)\\
=&-(-1)^{|f||y|}(-1)^{|f||x|}f\circ \pi(x)\pi(y)-(-1)^{|x||y|}(-1)^{|f||x|}(-1)^{|f||y|}f\circ \pi(y)\circ \pi(x)\\
=&(-1)^{|f|(|x|+|y|)}f\circ(-\pi(x)\pi(y)-(-1)^{|x||y|}\pi(y)\pi(x))\\
=&(-1)^{|f||x\bullet y|}f\circ \pi(x\bullet y)\\
=&\pi^*(x\bullet y)(f).\end{align*}
Then  $(\mathcal{V}^*,\pi^*)$ is a representation of $\mathcal{J}$ on $\mathcal{V}^{*}.$
\end{proof}
\begin{cor} \label{corep}
Let $(\mathcal{J},\bullet)$ be a mock-Lie superalgebra and $(\mathcal{J},L)$ be the adjoint representation of $\mathcal{J}$. Let us consider the even linear map $L^*:\mathcal{J}\rightarrow End(\mathcal{J}^*)$ defined by: $L^*(x)(f)=(-1)^{\vert f\vert \vert x\vert}f\circ L(x),~~\forall x\in \mathcal{J}_{\vert x\vert}, \, f\in \mathcal{J}^*_{\vert f\vert}$. Then, $(\mathcal{J}^*,L^*)$ is a representation of $(\mathcal{J},\bullet)$ on $\mathcal{J}^*$ called the coadjoint representation of $\mathcal{J}$.
\end{cor}
\quad In the following we introduce the notion of central extensions of mock-Lie superalgebras. An even linear map $\Omega:\mathcal{J}\times \mathcal{J}\rightarrow \mathcal{V}$ is called  a $2$-cocycle of $\mathcal{J}$ on a representation $(\mathcal{V},\pi)$ if 
 \begin{equation}\label{SuperCommCocycle}
     \Omega(x,y)=(-1)^{\vert x\vert \vert y\vert}\Omega(y,x),
 \end{equation}
\begin{equation}\label{CocycleCondition}
   \circlearrowleft_{x,y,z} \big((-1)^{\vert x\vert \vert z\vert}\Omega(x,y\bullet z)+(-1)^{\vert x\vert \vert z\vert}\pi(x)\Omega(y, z)\big)=0.\end{equation}

In this case, $(\mathfrak{A},\star_{\Omega})$ is called a central extension of $\mathcal{J}$ by $\mathcal{V}$ by means of $\Omega$. In the case where $\mathcal{V}=\mathbb{K},~\Omega$ is called an even scalar $2$-cocycle of $(\mathcal{J},\bullet)$. The vector
space of the even $2$-cocycles of $\mathcal{J}$ on the trivial $\mathcal{J}-$module $\mathcal{V}$ will be denoted by $(Z^2_{\rm MLie}(\mathcal{J},\mathcal{V}))_{\bar{0}}$.

Let $(\mathcal{J},\bullet)$ be a mock-Lie superalgebra, $\mathcal{V}$ be a $\mathbb{Z}_{2}$-graded vector space and $\Omega:\mathcal{J}\times \mathcal{J}\rightarrow \mathcal{V}$ be an even bilinear map. We define in the $\mathbb{Z}_{2}$-graded vector space $\mathfrak{J}:=\mathcal{J}\oplus \mathcal{V}$ the following product:\begin{equation}\label{TwistedSemiDirectProduct}
(x+u)\star_{\Omega}(y+v)=x\bullet y+\pi(x)(v)+(-1)^{\vert x\vert \vert y\vert}\pi(y)(u)+\Omega(x,y),    
\end{equation}
for all $ x,y\in \mathcal{J}_{\vert x\vert}\times \mathcal{J}_{\vert y\vert},~u,v\in \mathcal{V}.$
\begin{pro} \label{cent.ext}
With the above notation, 
the pair $(\mathfrak{J},\star_{\Omega})$ is a mock-Lie superalgebra if and only if $\Omega$ is a $2$-cocycle.
\end{pro}
\begin{proof} By supercommutativity of ''$\bullet$'' and $\Omega$, we can check that ''$\star_{\Omega}$'' is supercommutative. $(\mathfrak{J},\star_{\Omega})$ is a mock-Lie superalgebra. According to \eqref{S.Commt.Ident} and \eqref{RepCond},
 for all $x,y,z \in \mathfrak{J}$ and for all $u,v,w \in \mathcal(V):$
 \begin{align*}
     &(-1)^{|x||z|}(x+u)\star_{\Omega}\big((y+v)\star_{\Omega}(z+w)\big)+(-1)^{|y||x|}(y+v)\star_{\Omega}\big((z+w)\star_{\Omega}(x+u)\big)\\
     &+(-1)^{|z||y|}(z+w)\star_{\Omega}\big((x+u)\star_{\Omega}(y+v)\big)\\
     =&(-1)^{|x||z|}(x+u)\star_{\Omega}\big(y\bullet z+\pi(y)w+(-1)^{|z||y|}\pi(z)v+\Omega(y,z)\big)\\
     &+(-1)^{|y||x|}(y+v)\star_{\Omega}\big(z\bullet x+\pi(z)u+(-1)^{|z||x|}\pi(x)w+\Omega(z,x)\big)\\
     &+(-1)^{|z||y|}(z+w)\star_{\Omega}\big(x\bullet y+\pi(x)v+(-1)^{|y||x|}\pi(y)u+\Omega(x,y)\big)\\
    = & (-1)^{|x||z|}\big(x\bullet (y\bullet z)+\pi(x)\pi(y)w+(-1)^{|z||y|}\pi(x)\pi(z)v+(-1)^{(|y|+|z|)|x|}\pi(y\bullet z)u+\pi(x)\Omega(y,z)\\&+\Omega(x,y\bullet z)\big)+(-1)^{|y||x|}\big(y\bullet (z\bullet x)+\pi(y)\pi(z)u+(-1)^{|z||x|}\pi(y)\pi(x)w+(-1)^{(|z|+|x|)|y|}\pi(z\bullet x)v\\
    &+\pi(y)\Omega(z,x)+\Omega(y,z\bullet x)\big)+(-1)^{|z||y|}\big(z\bullet (x\bullet y)+\pi(z)\pi(x)v+(-1)^{|y||x|}\pi(z)\pi(y)u\\&+(-1)^{(|x|+|y|)|z|}\pi(x\bullet y)w+\pi(z)\Omega(x,y)+\Omega(z,x\bullet y)\big)\\
    =&(-1)^{|x||z|}x\bullet (y\bullet z)+(-1)^{|y||x|}y\bullet (z\bullet x)+(-1)^{|z||y|}z\bullet (x\bullet y)\\
 &+\big(\pi(x\bullet y)+\pi(x)\pi(y)+(-1)^{|y||x|}\pi(y)\pi(x)\big)w+(-1)^{|x||z|}\pi(x)\Omega(y,z)+(-1)^{|x||z|}\Omega(x,y\bullet z)\\
&+\big(\pi(z\bullet x)+\pi(z)\pi(x)+(-1)^{|x||z|}\pi(x)\pi(z)\big)v+(-1)^{|y||x|}\pi(y)\Omega(z,x)+(-1)^{|y||x|}\Omega(y,z\bullet x)\\
&+\big(\pi(y\bullet z)+\pi(y)\pi(z)+(-1)^{|z||y|}\pi(z)\pi(y)\big)u+(-1)^{|z||y|}\pi(z)\Omega(x,y)+(-1)^{|z||y|}\Omega(z,x\bullet y)\\
=&(-1)^{|x||z|}\pi(x)\Omega(y,z)+(-1)^{|x||z|}\Omega(x,y\bullet z)+(-1)^{|y||x|}\pi(y)\Omega(z,x)+(-1)^{|y||x|}\Omega(y,z\bullet x)\\&+(-1)^{|z||y|}\pi(z)\Omega(x,y)+(-1)^{|z||y|}\Omega(z,x\bullet y).
\end{align*}
Then 
$(\mathfrak{J},\star_{\Omega})$ is a mock-Lie superalgebra if and only if $\Omega$ is a $2$-cocycle.
\end{proof}

\section{Pseudo-euclidean mock-Lie superalgebras}\label{section.3}
In this Section, we introduce the notion of mock-Lie superalgebra with an even supersymmetric and non-degenerate bilinear form called pseudo-euclidean  mock-Lie superalgebra and give some constructions and characterizations.
\begin{defi}
Let $(\mathcal{A},\cdot)$ be a non-associative superalgebra. A bilinear form $B$ on $\mathcal{A}$ is 
\begin{enumerate}
\item[\bf (i)] supersymmetric if $B(x,y)=(-1)^{\vert x\vert \vert y\vert} B(y,x)$, for all $x\in \mathcal{A}_{\vert x\vert}$ and $y\in \mathcal{A}_{\vert y\vert}$;
\item[\bf (ii)] skew-supersymmetric if $B(x,y)=-(-1)^{\vert x\vert \vert y\vert} B(y,x)$, for all $x\in \mathcal{A}_{\vert x\vert}$ and $y\in \mathcal{A}_{\vert y\vert}$;
\item[\bf (ii)] non-degenerate if $x \in \mathcal{A}$ satisfies $B(x,y)=0$, for any $y \in \mathcal{A}$, then $x=0$;
\item[\bf (iii)] even if $B(\mathcal{A}_{\alpha}, \mathcal{A}_{\beta})=\{0\}$, where $(\alpha, \beta)\in {\mathbb Z}_2\times {\mathbb Z}_2$ with $\alpha\neq \beta$.
\end{enumerate}
\end{defi}

\begin{defi}\label{pseudo-RiemannianSupe.alg}
Let $(\mathcal{A},\cdot)$ be a non-associative superalgebra and $B:\mathcal{A} \times \mathcal{A}\rightarrow \mathbb{K}$ be an even supersymmetric and non-degenerate bilinear form.
Then, we say that $\mathcal{A}$ admits a pseudo-riemannian metric (or structure) $B$ and $(\mathcal{A},\cdot,B)$ is
termed pseudo-riemannian non-associative superalgebra.
\end{defi}
\begin{defi}
A mock-Lie superalgebra $(\mathcal{J},\bullet)$ is called pseudo-euclidean if it is provided with a pseudo-riemannian structure $B$ which is invariant (or associative), i.e. $B(x\bullet y,z) = B(x,y\bullet z),~\forall x,y,z\in \mathcal{J}$. It is denoted by $(\mathcal{J},B)$
and $B$ is called an invariant scalar product on $\mathcal{J}$.
\end{defi}

Now, we are going to give some properties of pseudo-euclidean mock-Lie superalgebras.
\begin{defi} 
Let $(\mathcal{J},\bullet,B)$ be a pseudo-euclidean mock-Lie superalgebra and $\mathcal{I}$ be a graded ideal of $\mathcal{J}$. Then, we say that
\begin{itemize}
    \item [1)]$\mathcal{I}$ is non-degenerate if $B|_{\mathcal{I}\times \mathcal{I}}$ is non-degenerate.
    \item[2)] $\mathcal{I}$ is isotropic if $B(\mathcal{I},\mathcal{I})=0.$
\end{itemize}
\end{defi}
\begin{defi}
Let $(\mathcal{J},\bullet,B)$ be a pseudo-euclidean mock-Lie superalgebra. We say that $(\mathcal{J},\bullet,B)$ is $B-$irreductible if $\mathcal{J}$ contains no non-trivial non-degenerate graded ideal. 
\end{defi}
\begin{defi}
     Let $(\mathcal{J},\bullet,B)$ be a pseudo-euclidean mock-Lie superalgebra. $Ann(\mathcal{J})=\{x\in \mathcal{J}/~x\bullet y=0,~\forall y\in \mathcal{J}\}$ is called the annihilator of $\mathcal{J}.$ 
\end{defi}
\begin{pro}
 Let $(\mathcal{J},\bullet,B)$ be a pseudo-euclidean mock-Lie superalgebra and $\mathcal{I}$ be a graded ideal of $\mathcal{J}$. Then,
\begin{itemize}
\item[\bf (a)] $\mathcal{I}^{\bot}= \{x\in \mathcal{J},~B(x,u)= 0,~\forall u\in \mathcal{I}\}$ is a graded ideal of $\mathcal{J}$ and $\mathcal{I}\bullet \mathcal{I}^{\bot}=\mathcal{I}^{\bot}\bullet \mathcal{I}=\{0\}$. Moreover $\mathcal{I}^{\bot}\subseteq Ann_\mathcal{J}(\mathcal{I})=\{x\in \mathcal{J}/ x\bullet\mathcal{I}=0\}. $
\item[\bf (b)] If $\mathcal{I}$ is a non-degenerate graded ideal of $\mathcal{J}$, then $\mathcal{J}=\mathcal{I}\oplus \mathcal{I}^{\bot}$.
\end{itemize}
\end{pro}
\begin{proof}
 ${\bf(a)}~$ Let $x\in \mathcal{I}^{\bot},~y\in \mathcal{J},~z\in \mathcal{I}$. By the associativity of $B$, we have $B(x\bullet y,z)=B(x,y\bullet z)=0$. So, $x\bullet y\in \mathcal{I}^{\bot}$. It follows that $\mathcal{I}^{\bot}$ is an ideal of $\mathcal{J}$. Since $\mathcal{I}$ is graded and $B$ is even, then $\mathcal{I}^{\bot}$ is graded. Moreover, $B(z\bullet x,y)=(-1)^{\vert x\vert \vert z\vert}B(x\bullet z,y)=(-1)^{\vert x\vert \vert z\vert}B(x,z\bullet y)=0$ because $B$ is non-degenerate, then $x\bullet z=z\bullet x=0$. Thus, $\mathcal{I}\bullet \mathcal{I}^{\bot}=\mathcal{I}^{\bot}\bullet \mathcal{I}=\{0\}$. Now let $y\in \mathcal{I}, z\in \mathcal{J}$ and $x\in \mathcal{I}^{\bot}$, then 
 $B(x\bullet y,z)=B(x,y\bullet z)=0.$
 Since $B$ is non-degenrate then $x\bullet y=0,$ then $x\in Ann_\mathcal{J}(\mathcal{I}).$
 The assertion ${\bf (b)}$ is clear.
 \end{proof}
 \begin{pro}
   Let $(\mathcal{J},\bullet,B)$ be a pseudo-euclidean mock-Lie superalgebra, and $B$ be even.
  \begin{itemize}
\item[\bf (a)] If $ \mathcal{F}=\{x\in \mathcal{J}_{\bar{1}}/x\bullet \mathcal{J}_{\bar{1}}=0\}.$ Then $\mathcal{F}\subseteq Ann(\mathcal{J}).$ 
\item[\bf (b)] If $(\mathcal{J},\bullet,B)$ is $B-$irreductible and $\mathcal{J}_{\bar{0}}\not =0$, then $\mathcal{J}_{\bar{1}}\bullet\mathcal{J}_{\bar{1}}=0$ iff $\mathcal{J}_{\bar{1}}=0.$
\end{itemize}
 \end{pro}
 \begin{proof}
  ${\bf(a)}~$ Let $x\in \mathcal{F}, y\in\mathcal{J_{\bar{0}}}$ and let $z\in \mathcal{J}_{\bar{1}}$, then $y\bullet z \in \mathcal{J}_{\bar{1}}$, and we have $$B(x\bullet y,z)=B(x,y\bullet z)=0$$
  Now let $z\in\mathcal{J}_{\bar{0}},$ then $y\bullet z \in \mathcal{J}_{\bar{0}}$ and $B$ even, so we have $$B(x\bullet y,z)=B(x,y\bullet z)=0.$$
  Then $B(x\bullet y, \mathcal{J})=0$, hence $x\bullet y\in \mathcal{J}^{\bot}=0.$ So $x\in Ann_\mathcal{J}.$\\
  ${\bf(b)}~$ If $\mathcal{J}_{\bar{1}}\bullet\mathcal{J}_{\bar{1}}=0$, by ${\bf(a)}~$ $\mathcal{J}_{\bar{1}}\subseteq Ann_\mathcal{J}$. Then $\mathcal{J}_{\bar{1}}$ is a graded ideal of $\mathcal{J}$. Let $x\in \mathcal{J}_{\bar{1}}$ such that $B(x,\mathcal{J}_{\bar{1}})=0.$ Since $B$ is even then $B(x,\mathcal{J}_{\bar{0}})=0$. So $B(x,\mathcal{J})=0$, since $B$ is non-degenerate then $x=0.$
 \end{proof}
 \begin{pro}
Let $(\mathcal{J},\bullet)$ be a mock-Lie superalgebra. Then, $\mathcal{J}^{2^{\perp}}=$Ann$(\mathcal{J})$, where $\mathcal{J}^{2^{\perp}}:=\{x\in \mathcal{J}/~B(x,\mathcal{J}^{2})=\{0\}\}$.
\end{pro}
\begin{proof}
    Let $x\in\mathcal{J}^{2^{\perp}},y\in \mathcal{J}.$ Since $B$ is invariant then $B(x\bullet y,z)=B(x,y\bullet z)=0,\quad\forall z\in\mathcal{J}.$ Since $B$ is non degenerate then $x\bullet y =0,$ Hence $x\in Ann(\mathcal{J}),$ and then $\mathcal{J}^{2^{\perp}}\subset Ann(\mathcal{J}).$ Conversely, it is clear that the fact that $B$ is invariant implies that $Ann(\mathcal{J)}\subset\mathcal{J}^{2^{\perp}}$. Then  $Ann(\mathcal{J)}=\mathcal{J}^{2^{\perp}}$
\end{proof}
In the following, we give some characterizations of pseudo-euclidean mock-Lie superalgebras.
\begin{pro}
Let $(\mathcal{J}=\mathcal{J}_{\bar{0}}\oplus \mathcal{J}_{\bar{1}},\bullet)$ be a mock-Lie superalgebra. Then, $\mathcal{J}$ is pseudo-euclidean if and only if the adjoint and the co-adjoint representations of $\mathcal{J}$ are equivalent and $dim(\mathcal{J}_{\bar{1}})$ is even.
\end{pro}
\begin{proof}
The proof is similar to that in \cite{HS04}.
\end{proof}
%\begin{ex}
    %Let $\mathcal{J}$ be a mock-Lie superalgebra defined in Example \ref{exp2.3}, let $B$ be an even supersymmetric bilinear form defined by: $B(\mathfrak{e},\mathfrak{e})=B(\mathfrak{o}_1,\mathfrak{o}_2)=1,$ and $B(\mathfrak{o}_1,\mathfrak{o}_1)=B(\mathfrak{o}_2,\mathfrak{o}_2)=0.$ Then $(\mathcal{J},B)$ is a pseudo-euclidean mock-Lie superalgebra.
%\end{ex}
\begin{pro}\label{ExistOddCentElmnt}
Let $(\mathcal{J}=\mathcal{J}_{\bar{0}}\oplus \mathcal{J}_{\bar{1}},\bullet,B)$ be a pseudo-euclidean Jordan superalgebra such that $\mathcal{J}_{\bar{0}}$ is a nilpotent Jordan algebra and $\mathcal{J}_{\bar{1}}\neq \{0\}$. Then $Ann(\mathcal{J})\cap \mathcal{J}_{\bar{1}}\neq \{0\}$.
\end{pro}
\begin{proof}
Since $\mathcal{J}_{\bar{0}}$ is a nilpotent Jordan algebra and $\mathcal{J}_{\bar{1}}\neq \{0\}$, then there exists a non-zero element $X_{1}\in \mathcal{J}_{\bar{1}}$ such that $\mathcal{J}_{\bar{0}}\bullet X_{1}=\{0\}$. Moreover, the fact that $B$ is invariant implies that $\{0\}=B(\mathcal{J}_{\bar{0}}\bullet X_{1},\mathcal{J}_{\bar{1}})=
B(\mathcal{J}_{\bar{0}}, X_{1}\bullet \mathcal{J}_{\bar{1}})$. So, $X_{1}\bullet \mathcal{J}_{\bar{1}}=\{0\}$ because $B$ is even and non-degenerate. Consequently, $X_{1}\in$Ann$(\mathcal{J})\cap \mathcal{J}_{\bar{1}}$.
\end{proof}
Now, recall that in \cite{BakBen1SymJ.J.A} it has been proved that every mock-Lie algebra $(\mathcal{J},\bullet)$ is a nilpotent Jordan algebra. Consequently, invoking Proposition \ref{ExistOddCentElmnt} and Proposition \ref{Jacob-Jor.Is.Jor}, we get the following corollary
which will be very useful later on.
\begin{cor}\label{central-element}
Let $(\mathcal{J}=\mathcal{J}_{\bar{0}}\oplus \mathcal{J}_{\bar{1}},\bullet,B)$ be a pseudo-euclidean mock-Lie superalgebra such that  $\mathcal{J}_{\bar{1}}\neq \{0\}$. Then $Ann(\mathcal{J})\cap \mathcal{J}_{\bar{1}}\neq \{0\}$.
\end{cor}
Our next step is to introduce the concept of $T^*-$extension briefly, which will be utilized to construct a bilinear form. . By the corollary \eqref{corep}, and the proposition \ref{cent.ext}, $\mathcal{J}\oplus\mathcal{J}^*$ is a mock-Lie superalgebra, $T^*-$extension of $\mathcal{J},$ with the product 
\begin{equation*}
(x+u)\star_{\Omega}(y+v)=x\bullet y+L^*(x)(v)+(-1)^{\vert x\vert \vert y\vert}L^*(y)(u)+\Omega(x,y).    
\end{equation*}
Consider the bilinear form $B$ in $\mathcal{J}\oplus\mathcal{J}^*$, for all $x,y\in \mathcal{J}_{\vert x\vert}\times \mathcal{J}_{\vert y\vert},~f,g\in \mathcal{J}^*,$
$$B(x+f,y+g)=f(y)+(-1)^{|x||y|}g(x).$$
Then we have the following proposition. 
\begin{pro}
    $(\mathcal{J}\oplus\mathcal{J}^*,B)$ is a pseudo-euclidean mock-Lie superalgebra if and only if $B$ is non-degenerate and $\Omega$ verifying
    \begin{equation}\label{supercycl-omg}
        \Omega(x,y)(z)=(-1)^{|x|(|y|+|z|)}\Omega(y,z)(x),
    \end{equation}
    forall $x,y,z\in \mathcal{J}_{\vert x\vert}\times \mathcal{J}_{\vert y\vert}\times\mathcal{J}_{\vert z\vert}.$
\end{pro}
\begin{proof}
    Clearly $B$ is supersymmetric. Let $x,y,z\in \mathcal{J}_{\vert x\vert}\times \mathcal{J}_{\vert y\vert}\times\mathcal{J}_{\vert z\vert}, f,g,h\in \mathcal{J}^*_{\vert f\vert}\times \mathcal{J}^*_{\vert g\vert}\times\mathcal{J}^*_{\vert h\vert}$
    \begin{align*}
        &B\big((x+f)\bullet_{\Omega}(y+g),z+h\big)\\
        &=B\big(x\bullet y+L^*(x)(g)+(-1)^{|x||y|}L^*(y)(f)+\Omega(x,y),z+h\big)\\
        &=L^*(x)(g)(z)+(-1)^{|x||y|}L^*(y)(f)+\Omega(x,y)(z)+(-1)^{|z|(|x|+|y|)}h(x\bullet y)\\
        &=(-1)^{|x||y|}g(x\bullet y)+f(y\bullet z)+(-1)^{|z|(|x|+|y|)}h(x\bullet y)+\Omega(x,y)(z).
    \end{align*}
    On the other hand,
    \begin{align*}
       &B\big((x+f),(y+g)\bullet_{\Omega}(z+h)\big)\\
       &=B\big(x+f,y\bullet z+L^*(y)(h)+(-1)^{|y||z|}L^*(z)(g)+\Omega(y,z)\big)\\
       &=f(y\bullet z)+(-1)^{|x|(|y|+|z|)}L^*(y)(h)(x)+(-1)^{|y||z|+|x|(|y|+|z|)}L^*(z)(g)(x)+(-1)^{|x|(|y|+|z|)}\Omega(y,z)\\
       &=f(y\bullet z)+(-1)^{|x|(|y|+|z|)+|y||z|}h(y\bullet x)+(-1)^{|x|(|y|+|z|)}g(z\bullet x)+(-1)^{|x|(|y|+|z|)}\Omega(y,z)\\
       &=f(y\bullet z)+(-1)^{|z|(|x|+|y|)}h(x\bullet y)+(-1)^{|x||y|}g(x\bullet y)+(-1)^{|x|(|y|+|z|)}\Omega(y,z).
    \end{align*}
    Then $B$ is invariant if and only if \eqref{supercycl-omg} holds.
\end{proof}

\section{Extensions of pseudo-euclidean mock-Lie superalgebras}\label{section.4}
In this section, we study double extensions of pseudo-euclidean mock-Lie superalgebras. And we introduce the notion of generalized double
extensions of mock-Lie superalgebra by the one dimentional odd mock-Lie superalgebra. 
\subsection{Doubles extensions}
\begin{defi}
    Let $(\mathcal{J},\bullet,B)$ be a pseudo-euclidean mock-Lie superalgebra and $D\in \mathfrak{AnDer}(\mathcal{J})$ of degree $\alpha$. We say that $D$ is supersymmetric with respect to $B$ if 
    $$B\big(D(x),y\big)=(-1)^{\alpha|x|}B\big(x,D(y)\big),\quad \forall x,y\in \mathcal{J}_{|x|}\times\mathcal{J}.$$ 
    The set of all supersymmetric anti-superderivations of $\mathcal{J}$ with respect to $B$ well be denoted $\mathfrak{AnDer}_s(\mathcal{J}).$
\end{defi}
\begin{thm}\label{centralex}
 Let $(\mathcal{J}_{1},\bullet_{1},B)$ be a pseudo-euclidean mock-Lie superalgebra, and $\phi:\mathcal{J}_{2}\rightarrow \mathfrak{AnDer}_{s}(\mathcal{J}_{1})$ a representation of mock-Lie superalgebras. Let $\varphi$ be the even map defined on $\mathcal{J}_{1}\times \mathcal{J}_{1}$ in ${\mathcal{J}_2}^{*}$ by $$\varphi(x,y)(a)=(-1)^{|x|(|y|+|a|)}B\big(y,\phi(a)(x)\big),$$  for all $x,y\in {{\mathcal{J}_{1}}_{|x|}}\times {{\mathcal{J}_{1}}_{|y|}}$ and $a\in {\mathcal{J}_{2}}_{|a|}$. Then $\mathcal{J}_{1}\oplus {\mathcal{J}_{2}}^*$ with the product 
 $$(x+f)\bullet_c(y+g)=x\bullet_1 y+\varphi(x,y),$$
 is a mock-Lie superalgebra. And $\phi$ can be extended to representation $\tilde{\phi}$ of $\mathcal{J}_{2}$ in $\mathfrak{AnDer}(\mathcal{J}_{1}\oplus {\mathcal{J}_{2}}^*)$ defined by 
 $$\tilde{\phi}(a)(x+f)=\phi(a)(x)+{L_2}^*(a)(f).$$
 Where ${L_2}^*$ is the coadjoint representation of $\mathcal{J}_2.$ 
\end{thm}
\begin{proof}
Let $x,y\in {\mathcal{J}_{1}}_{|x|}\times{\mathcal{J}_{1}}_{|y|}$, and $a\in{\mathcal{J}_{2}}_{|a|}$,
\begin{align*}
    \varphi(y,x)(a)&=(-1)^{|y|(|x|+|a|)}B\big(x,\phi(a)(y)\big)\\
    &=(-1)^{|y|(|x|+|a|)+|x||a|}B\big(\phi(a)(x),y\big)\\
    &=(-1)^{|x||y|+|a|(|x|+|y|)+|y|(|a|+|x|)}B\big(y,\phi(a)(x)\big)\\
    &=(-1)^{|x||y|+|a||x|+|x||y|}B\big(y,\phi(a)(x)\big)\\
    &=(-1)^{|x||y|}\Big((-1)^{|x|(|y|+|a|)}B\big(y,\phi(a)(x)\big)\Big)\\
    &=(-1)^{|x||y|}\varphi(x,y)(a).
\end{align*}
Then the product $\bullet_c$ is super-symmetric in the sense $$(x+f)\bullet_c(y+g)=(-1)^{|x||y|}(y+g)\bullet_c(x+f).$$
Let $x,y,z\in{\mathcal{J}_{1}}_{|x|}\times{\mathcal{J}_{1}}_{|y|}\times{\mathcal{J}_{1}}_{|z|}$, and $f,g,h\in {\mathcal{J}_{2}}^*_{|f|}\times{\mathcal{J}_{2}}^*_{|g|}\times{\mathcal{J}_{2}}_{|h|}$, we have \\
\begin{align*}
&\circlearrowleft_{x,y,z}(-1)^{|x||z|}\Big((x+f)\bullet_c\big((y+g)\bullet_c(z+h)\big)\Big)\\&= \circlearrowleft_{x,y,z}(-1)^{|x||z|}x\bullet_1(y\bullet_1z)+ \circlearrowleft_{x,y,z}(-1)^{|x||z|}\varphi(x,y\bullet_1z).
\end{align*}
Let $x,y\in {\mathcal{J}_{1}}_{|x|}\times{\mathcal{J}_{1}}_{|y|}$,and $a\in{\mathcal{J}_{2}}_{|a|}$,
\begin{align*}
    (-1)^{|x||z|}\varphi(x,y\bullet_1z)(a)
    &=(-1)^{|x||z|+|x|(|y|+|z|+|a|)}B\big(y\bullet_1z,\phi(a)(x)\big)\\
    &=(-1)^{|x||y|+|x||a|}B\big(y\bullet_1z,\phi(a)(x)\big)\\
    &=(-1)^{|x|(|y|+|a|)+|a|(|y|+|z|)}B\big(\phi(a)(y\bullet_1z),(x)\big)\\
    &=(-1)^{|x|(|y|+|a|)+|a|(|y|+|z|)}B\big(-\phi(a)(y)\bullet_1z-(-1)^{|a||y|}y\bullet_1\phi(a)(z),x\big)\\
    &=-(-1)^{|x|(|y|+|a|)+|a|(|y|+|z|)}B\big(\phi(a)(y)\bullet_1z,x\big)\\
    &-(-1)^{|x|(|y|+|a|)+|a|(|y|+|z|)+|a||y|}B\big(y\bullet_1\phi(a)(z)\big)\\
    &=-(-1)^{|x|(|y|+|a|)+|a|(|y|+|z|)}B\big(\phi(a)(y),z\bullet_1x\big)\\
    &-(-1)^{|x|(|y|+|a|)+|a||z|}B\big(y\bullet_1\phi(a)(z),x\big)\\
    &=-(-1)^{|y|(|z|+|a|)}B\big(z\bullet_1x,\phi(a)(y)\big)\\
    &-(-1)^{|x|(|y|+|a|)+|a||z|+|y|(|a|+|z|)}B\big(\phi(a)(z)\bullet_1y,x\big)\\
    &=-(-1)^{|x||y|}\varphi(y,z\bullet_1x)-(-1)^{|x|(|y|+|a|)+|a||z|+|y|(|a|+|z|)}B\big(\phi(a)(z),y\bullet_1x\big)\\
    &=-(-1)^{|x||y|}\varphi(y,z\bullet_1x)-(-1)^{|a|(|x|+|z|)+|y|(|a|+|z|)}B\big(\phi(a)(z),x\bullet_1y\big)\\
    &=-(-1)^{|x||y|}\varphi(y,z\bullet_1x)-(-1)^{|z||y|}\varphi(z,x\bullet_1y).   
\end{align*}
By the super-Jacobi identity of ${\mathcal{J}_{1}}$ and the last identity $\mathcal{J}_{1}\oplus {\mathcal{J}_{2}}^*$ is mock-Lie superalgebra.
Now let $x,y\in {\mathcal{J}_{1}}_{|x|}\times{\mathcal{J}_{1}}_{|y|},f,g\in{\mathcal{J}_{2}}^*_{|f|}\times{\mathcal{J}_{2}}^*_{|g|}$, and $a\in{\mathcal{J}_{2}}_{|a|}$\\
\begin{align*}
    &\tilde{\phi}(a)\big((x+f)\bullet_c(y+g)\big)=\tilde{\phi}(a)(x\bullet_1y+\varphi(x,y))={\phi}(a)(x\bullet_1y)+{L_2}^*(a)(\varphi(x,y)),
\end{align*}
\begin{align*}
    &\tilde{\phi}(a)(x+f)\bullet_c(y+g)=({\phi}(a)(x)+{L_2}^*(a)(f))\bullet_c(y+g)={\phi}(a)\bullet_1y+\varphi(\phi(a)(x),y)
\end{align*}
and
\begin{align*}
    &(x+f)\bullet_c\tilde{\phi}(a)(y+g)=(x+f)\bullet_c(\phi(a)(y)+{L_2}^*(a)(g)=x\bullet_1\phi(a)(y)+\varphi(x,\phi(a)(y)).
\end{align*}
Then $\tilde{\phi}$ is a super antiderivation if and only if $${L_2}^*(a)(\varphi(x,y))=-\varphi(\phi(a)(x),y)-(-1)^{|a||x|}\varphi(x,\phi(a)(y)).$$
Let $b\in {\mathcal{J}_{2}}$
\begin{align*}
    &{L_2}^*(a)(\varphi(x,y))(b)\\
    &=(-1)^{|a|(|x|+|y|)}\varphi(x,y)(a\bullet_2b)\\
    &=(-1)^{|a|(|x|+|y|)+|x|(|y|+|a|+|b|)}B\big(y,\phi(a\bullet_2b)(x)\big)\\
    &=(-1)^{|a|(|x|+|y|)+|x|(|y|+|a|+|b|)}B\big(y,-\phi(a)\phi(b)(x)-(-1)^{|b||a|}\phi(b)\phi(a)(x)\big)\\
    &=-(-1)^{|a|(|x|+|y|)+|x|(|y|+|a|+|b|)}B\big(y,\phi(a)\phi(b)(x)\big)\\&-(-1)^{|a|(|x|+|y|)+|x|(|y|+|a|+|b|)}B\big(y,(-1)^{|b||a|}\phi(b)\phi(a)(x)\big)\\
    &=-(-1)^{|a||x|+|x|(|a|+|y|+|b|)}B\big(\phi(a)(y),\phi(b)(x)\big)-(-1)^{(|x|+|a|)(|y|+|b|)}B\big(y,\phi(b)\phi(a)(x)\big)\\
    &=-(-1)^{|a||x|}\varphi(x,\phi(a)(y))(b)-\varphi(\phi(a)(x),y)(b).
    \end{align*}
Then $\tilde{\phi}$ is a anti superderivation. Clearly $\tilde{\phi}$ is a sum of tow representations then $\tilde{\phi}$ is a representation.
\end{proof}
\begin{lem}\label{prodint}
Let $(\mathcal{J}_1,\bullet_1)$ and $(\mathcal{J}_2,\bullet_2)$ be tow  mock-Lie superalgebras, and $\phi:J_{2}\rightarrow {AnDer}(\mathcal{J}_{1})$ an action of mock-Lie superalgebras. Then $\mathcal{J}_{1}\oplus {\mathcal{J}_{2}}$ is a mock-Lie superalgebras with the product 
$$(x+a) \bullet_{i}(y+b)=x\bullet_{1}y+\phi(a)(y)+ (-1)^{|x||y|}\phi(b)(x)+a\bullet_{2}b,$$
for all $ x,y\in {{\mathcal{J}_{1}}_{|x|}}\times {{\mathcal{J}_{1}}_{|y|}},a,b\in {{\mathcal{J}_{2}}_{|a|}}\times {{\mathcal{J}_{2}}_{|b|}}$.
Note that $|x|=|a|, |y|=|b|,|z|=|c|.$    The triple $(\mathcal{J}_2,\bullet_2,\pi)$ is called an action of $\mathcal J_1$ on $\mathcal J_2$. 
\end{lem}
\begin{proof}
   Let $ x,y\in {{\mathcal{J}_{1}}_{|x|}}\times {{\mathcal{J}_{1}}_{|y|}},a,b\in {{\mathcal{J}_{2}}_{|a|}}\times {{\mathcal{J}_{2}}_{|b|}}$.
   \begin{align*}
       &(-1)^{|x||z|}(x+a)\bullet_i\big((y+b)\bullet_i(z+c)\big)\\
       &=(-1)^{|x||z|}x\bullet_1(y\bullet_1z)+(-1)^{|x||z|}a\bullet_2(b\bullet_2c)+(-1)^{|x||z|}x\bullet_1\phi(b)(z)+(-1)^{|z|(|x|+|y|)}x\bullet_1\phi(c)(y)\\
       &+(-1)^{|x||z|}\phi(a)(y\bullet_1z)+(-1)^{|x||z|}\phi(a)\phi(b)(z)+(-1)^{|z|(|x|+|y|)}\phi(a)\phi(c)(y)+(-1)^{|y||x|}\phi(b\bullet_2c)(x),
   \end{align*}
    \begin{align*}
       &(-1)^{|y||x|}(y+b)\bullet_i\big((z+c)\bullet_i(x+a)\big)\\
       &=(-1)^{|y||x|}y\bullet_1(z\bullet_1x)+(-1)^{|y||x|}b\bullet_2(c\bullet_2a)+(-1)^{|y||x|}y\bullet_1\phi(c)(x)+(-1)^{|x|(|y|+|z|)}y\bullet_1\phi(a)(z)\\
       &+(-1)^{|y||x|}\phi(b)(z\bullet_1x)+(-1)^{|y||x|}\phi(b)\phi(c)(x)+(-1)^{|x|(|y|+|z|)}\phi(b)\phi(a)(z)+(-1)^{|z||y|}\phi(c\bullet_2a)(y).
   \end{align*}
   and
    \begin{align*}
       &(-1)^{|z||y|}(z+c)\bullet_i\big((x+a)\bullet_i(y+b)\big)\\
       &=(-1)^{|z||y|}z\bullet_1(x\bullet_1y)+(-1)^{|z||y|}c\bullet_2(a\bullet_2b)+(-1)^{|z||y|}z\bullet_1\phi(a)(y)+(-1)^{|y|(|z|+|x|)}z\bullet_1\phi(b)(x)\\
       &+(-1)^{|z||y|}\phi(c)(x\bullet_1y)+(-1)^{|z||y|}\phi(c)\phi(a)(y)+(-1)^{|y|(|z|+|x|)}\phi(c)\phi(b)(x)+(-1)^{|x||z|}\phi(a\bullet_2b)(z).
   \end{align*}
   By the super-Jacobi identity,
   \begin{align*}
      &(-1)^{|x||z|}x\bullet_1(y\bullet_1z)+ (-1)^{|y||x|}y\bullet_1(z\bullet_1x)+(-1)^{|z||y|}z\bullet_1(x\bullet_1y)=0,\\
      &(-1)^{|x||z|}a\bullet_2(b\bullet_2c)+(-1)^{|y||x|}b\bullet_2(c\bullet_2a)+(-1)^{|z||y|}c\bullet_2(a\bullet_2b)=0.
   \end{align*}
   Since $\phi$ is an anti superderivation,
   \begin{align*}
       &(-1)^{|z||y|}\phi(c)(x\bullet_1y)+(-1)^{|y||x|}y\bullet_1\phi(c)(x)+(-1)^{|z|(|x|+|y|)}x\bullet_1\phi(c)(y)=0,\\
       &(-1)^{|y||x|}\phi(b)(z\bullet_1x)+(-1)^{|x||z|}x\bullet_1\phi(b)(z)+(-1)^{|y|(|z|+|x|)}z\bullet_1\phi(b)(x)=0,\\
       &(-1)^{|x||z|}\phi(a)(y\bullet_1z)+(-1)^{|z||y|}z\bullet_1\phi(a)(y)+(-1)^{|x|(|y|+|z|)}y\bullet_1\phi(a)(z)=0.
   \end{align*}
   Since $\phi$ is a representation of mock-Lie superalgebras,
   \begin{align*}
     &(-1)^{|y||x|}\phi(b\bullet_2c)(x)+(-1)^{|y||x|}\phi(b)\phi(c)(x)+(-1)^{|y|(|z|+|x|)}\phi(c)\phi(b)(x)=0\\
     &(-1)^{|z||y|}\phi(c\bullet_2a)(y)+(-1)^{|z||y|}\phi(c)\phi(a)(y)+(-1)^{|z|(|x|+|y|)}\phi(a)\phi(c)(y)=0\\
     &(-1)^{|x||z|}\phi(a\bullet_2b)(z)+(-1)^{|x||z|}\phi(a)\phi(b)(z)+(-1)^{|x|(|y|+|z|)}\phi(b)\phi(a)(z)=0.
   \end{align*}
   Then the super-Jacobi identity holds.
\end{proof}
\begin{thm}\label{doubelex}
   With the hypothesis in the Theorem \ref{centralex}, $$ \mathcal{J}=\mathcal{J}_{2}\oplus \mathcal{J}_{1}\oplus{\mathcal{J}_{2}}^*,$$
 is a mock-Lie superalgebra  with the product  \begin{align*}
       (a+x+f)\bullet_d(b+y+g)=&a\bullet_{2} b+ x\bullet_{1} y+\phi(a)(y)+(-1)^{|x||y|}\phi(b)(x)\\&+{L_2}^*(a)(g)+(-1)^{|x||y|}{L_2}^*(b)(f)+\varphi(x,y)\end{align*} 
      for all $(a+x+f),(b+y+g)\in \mathcal{J}_{2}\oplus \mathcal{J}_{1}\oplus{\mathcal{J}_{2}}^*$. Note that  $|x|=|a|=|f|,|y|=|b|=|g|,|z|=|c|=|h|.$
   \end{thm}
   \begin{proof}
     By Theorem \ref{centralex}, the representation of mock-Lie superalgebras $\tilde{\phi}$ is a super antiderivation of ${\mathcal{J}_{1}}\oplus{\mathcal{J}_{2}}$, then 
      \begin{align*}
       (a+x+f)\bullet_d
       (b+y+g)=a\bullet_{2} b+ (x+f)\bullet_c(y+g)+\tilde{\phi}(a)(y)+(-1)^{|x||y|}\tilde{\phi(b)}(x),\end{align*}
       by Lemma \ref{prodint}, $\mathcal{J}=\mathcal{J}_{2}\oplus \mathcal{J}_{1}\oplus{\mathcal{J}_{2}}^*$ is a mock-Lie superalgebra.
   \end{proof}
   \begin{pro}
     Let $\mathcal{J}$  be a mock-Lie superalgebra of Theorem \ref{doubelex}, let $\s$ be a supersymmetric invariant bilinear form in $\mathcal{J}_2$. Then the bilinear form defined on $\mathcal{J}$ by $$\tilde{B}(a+x+f,b+y+g)=B(x,y)+\s(a,b)+f(b)+(-1)^{|x||y|}g(a),$$
     is a supersymmetric invariant bilnear form on $\mathcal{J}$.
   \end{pro}
   \begin{proof} It is clear that  $\tilde{B}$ is supersymmetric. Let $x,y,z\in {\mathcal{J}_{1}}_{|x|}\times{\mathcal{J}_{1}}_{|y|}\times{\mathcal{J}_{1}}_{|z|},f,g,h\in{\mathcal{J}_{2}}^*_{|f|}\times{\mathcal{J}_{2}}^*_{|g|}\times{\mathcal{J}_{2}}^*_{|h|},a,b,c\in\times{\mathcal{J}_{2}}_{|a|}\times{\mathcal{J}_{2}}_{|b|}\times\in{\mathcal{J}_{2}}_{|c|},$ 
    \begin{align*}
        &\tilde{B}\big((a+x+f)\bullet_d(b+y+g),c+z+h\big)\\
        &=\tilde{B}\big(a\bullet_2b+x\bullet_1y+\phi(a)(y)+(-1)^{|x||y|}\phi(b)(x)\\&+{L_2}^*(a)(g)+(-1)^{|x||y|}{L_2}^*(b)(f)+\varphi(x,y),c+z+h\big)\\
        &=B\big(x\bullet_1y+\phi(a)(y)+(-1)^{|x||y|}\phi(b)(x),z\big)+\sigma(a\bullet_2b,c)\\
        &+{L_2}^*(a)(g)(c)+(-1)^{|x||y|}{L_2}^*(b)(f)(c)+\varphi(x,y)(c)+(-1)^{|z|(|x|+|y|)}h(a\bullet_2b)\\
        &=B(x\bullet_1y,z)+B(\phi(a)(y),z)+(-1)^{|x||y|}B(\phi(b)(x),z))+\sigma(a\bullet_2b,c)\\
        &+(-1)^{|x||y|}g(a\bullet_2c)+f(b\bullet_2c)+(-1)^{|x|(|y|+|z|)}B(y,\phi(c)(x))+(-1)^{|z|(|x|+|y|)}h(a\bullet_2b).
    \end{align*}
    \begin{align*}
        &\tilde{B}\big(a+x+f,(b+y+g)\bullet_d(c+z+h)\big)\\
        &=\tilde{B}\big(a+x+f,b\bullet_2c+y\bullet_1z+\phi(b)(z)+(-1)^{|y||z|}\phi(c)(y)\\&+{L_2}^*(b)(h)+(-1)^{|y||z|}{L_2}^*(c)(g)+\varphi(y,z)\big)\\
        &=B\big(x,y\bullet_1z+\phi(b)(z)+(-1)^{|y||z|}\phi(c)(y)\big)+\sigma(a,b\bullet_2c)+f(b\bullet_2c)\\&+(-1)^{|x|(|y|+|z|)}{L_2}^*(b)(h)(a)+(-1)^{|x|(|z||y|)+|y||z|}{L_2}^*(g)(a)+(-1)^{|x|(|y|+|z|)}\varphi(y,z)(a)\\
        &=B(x,y\bullet_1z)+B\big(x,\phi(b)(z)\big)+(-1)^{|y||z|}B\big(x,\phi(c)(y)\big)+\sigma(a,b\bullet_2c)+f(b\bullet_2c)\\&+(-1)^{|x|(|y|+|z|)+|y||z|}h(b\bullet_2a)+(-1)^{|x|(|y|+|z|)}g(c\bullet_2a)+(-1)^{|x|(|y|+|z|)+|y|(|x|+|z|)}B\big(z,\phi(a)(y)\big).\\
        &=B(x,y\bullet_1z)+(-1)^{|x||y|}B\big(\phi(b)(x),z\big)+(-1)^{|x|(|y|+|z|)}B\big(y,\phi(c)(x)\big)+\sigma(a,b\bullet_2c)\\
        &+f(b\bullet_2c)+(-1)^{|z|(|x|+|y|)}h(a\bullet_2b)+(-1)^{|x||y|}g(a\bullet_2c)+B\big(\phi(a)(y),z\big).
    \end{align*}
    Then $\tilde{B}$ is invariant.
   \end{proof}
\begin{rmk} 
 Let $(\mathcal{J},\bullet)$ be a mock-Lie superalgebra and $u\in\mathcal{J}$. Then $\mathcal{U}=\mathbb{K}u$ with a null product is the unique one-dimensional mock-Lie superalgebra. 
\end{rmk}
\subsection{Generalized double extensions}\label{section.5}
\begin{defi}
    Let $\mathcal{J}$ be a mock-Lie superalgebra,  $D\in {\mathfrak{AnDer}(\mathcal{J})}_{\bar{1}}$ such that $D^2(x)=0,$ forall $x\in \mathcal{J},$ and $x_0\in Ann(\mathcal{J})\cap\mathcal{J}_{\bar{0}}.$ We say that $(D,x_0)$ is an admissible pair of $\mathcal{J}$ if $D(x_0)=0.$
\end{defi}
\begin{lem}\label{gen-semi-prod-by-one-dim}
    Let $(\mathcal{J},\bullet,B)$ be a pseudo-euclidean mock-Lie superalgebra, $\mathcal{U}=\mathcal{U}_{\bar{1}}=\mathbb{K}u$ be the one-dimensional odd mock-Lie superalgebra and $(D,x_0)$ an admissible pair of $\mathcal{J}.$ Then the vector space $\bar{\mathcal{J}}=\mathcal{U}\oplus\mathcal{J}$ is a mock-Lie superalbegra with the product  
    \begin{align*}
        &u\bar{\bullet}u=x_0,\\
        &u\bar{\bullet}x=D(x),\\
        &x\bar{\bullet}y=x\bullet y,\quad \forall x,y \in \mathcal{J}_{|x|}\times\mathcal{J}_{|y|}.
    \end{align*}
    $\bar{\mathcal{J}}$ is called the generalized semi-direct product of $\mathcal{J}$ by the one-dimensional odd mock-Lie superalgebra $\mathcal{U}$ by means $(D,x_0).$   
\end{lem}
\begin{proof}
    The supercommutativity is obvious. Let $X=\alpha u +x\in \bar{\mathcal{J}}_{|X|}$, $Y=\beta u+y\in \bar{\mathcal{J}}_{|Y|}$ and $Z=\gamma u +z\in \bar{\mathcal{J}}_{|Z|},$ such that $|X|=|x|=|u|,|Y|=|y|=|u|$ and $|Z|=|z|=|u|.$ We get
    \begin{align*}
        &(-1)^{|X||Z|}X\bar{\bullet}(Y\bar{\bullet}Z)=(-1)^{|X||Z|}\Big(x\bullet (y\bullet z)-\gamma x\bullet D(y)+\beta x \bullet D(z)\\&+\beta\gamma x\bullet x_0+\alpha D(y\bullet z)-\alpha \gamma D^2(y)+\alpha \beta D^2(z)+\alpha \beta \gamma D(x_0)\Big).\\
        &(-1)^{|Y||X|}Y\bar{\bullet}(Z\bar{\bullet}X)=(-1)^{|Y||X|}\Big(y\bullet (z\bullet x)-\alpha y\bullet D(z)+\gamma y \bullet D(x)\\&+\gamma\alpha y\bullet x_0+\beta D(z\bullet x)-\beta \alpha D^2(z)+\beta \gamma D^2(x)+\alpha \beta \gamma D(x_0)\Big).\\
        &(-1)^{|Z||Y|}Z\bar{\bullet}(X\bar{\bullet}Y)=(-1)^{|Z||Y|}\Big(z\bullet (x\bullet y)-\beta z\bullet D(x)+\alpha z \bullet D(y)\\&+\alpha\beta z\bullet x_0+\gamma D(x\bullet y)-\gamma \beta D^2(x)+\gamma \alpha D^2(y)+\alpha \beta \gamma D(x_0)\Big).
    \end{align*}
    Since $\mathcal{J}$ is a mock-Lie superalgebra,
    \begin{align*}
        &(-1)^{|X||Z|}(x\bullet (y\bullet z))+(-1)^{|Y||X|}(y\bullet (z\bullet x))+(-1)^{|Z||Y|}(z\bullet (x\bullet y))=0.
    \end{align*}
   Since $D\in \mathfrak{AnDer(\mathcal{J)}}_{\bar{1}}$, then 
    \begin{align*}
        &D(y\bullet z)-y\bullet D(z)+z \bullet D(y)=0,\\
        &D(z\bullet x)-z\bullet D(x)+x\bullet D(z)=0,\\
        &D(x\bullet y)-x\bullet D(y)+y\bullet D(x)=0.
    \end{align*}
    Since  $(D,x_0)$ is an admissible pair of $\mathcal{J}$, then we have $ D(x_0)=0$ and $x\bullet x_0=0,\quad\forall x\in \mathcal{J}.$
    Therefore $\bar{\mathcal{J}}$ is a mock-Lie superalgebra.
\end{proof}
\begin{lem}\label{central-by-one}
   Let $(\mathcal{J},\bullet,B)$ be a pseudo-euclidean mock-Lie superalgebra, $\mathcal{U}=\mathcal{U}_{\bar{1}}=\mathbb{K}u$ be the one-dimensional odd mock-Lie superalgebra, and $\mathcal{U}^*$ its dual space. Let $D\in(\mathfrak{AnDer}(\mathcal{J}))_{\bar{1}}$ and $\varphi$ be a even linear map defined on $\mathcal{J}\times\mathcal{J}$ in $\mathcal{U}^*$ by
   $$\varphi(x,y)=B\big(y,D(x)\big),\quad\forall x,y\in \mathcal{J}.$$
   Then $\hat{\mathcal{J}}=\mathcal{J}\oplus\mathcal{U}^*$ with the product
   $$x\hat{\bullet}y=x\bullet y + \varphi(x,y)u^*,\quad \forall x,y\in\mathcal{J}\times\mathcal{J},$$
   is a mock-Lie superalgebra, central extension of $\mathcal{J}$ by $\mathcal{U}^*.$
\end{lem}
\begin{proof}
    The proof is similar to the proof of Theorem \ref{centralex}.
\end{proof}
\begin{lem}\label{ext-pair}
 Let $(\mathcal{J},\bullet,B)$ be a pseudo-euclidean mock-Lie superalgebra, $\mathcal{U}=\mathcal{U}_{\bar{1}}=\mathbb{K}u$ be the one-dimensional odd mock-Lie superalbegra, and $\mathcal{U}^*$ its dual space. Let $(D,x_0)$ be an admissible pair of $\mathcal{J}$ such that $D$ is super symmetric with respect to $B$, and $B(x_0,x_0)=0.$ Then \begin{align*}
     &\hat{x}_0= x_0+\lambda u^*,\quad \lambda\in\mathbb{K},\\
     &\hat{D}(x+\alpha u^*)=D(x)-B(x_0,x)u^*,\quad \forall (x+\alpha u^*)\in\hat{\mathcal{J}},
 \end{align*} 
 is an admissible pair of $\hat{\mathcal{J}}.$
\end{lem}
\begin{proof}
    Let $x+\alpha u^*,\; y+\beta u^*\in\hat{\mathcal{J}}$, then we have
    \begin{align*}
        &\hat{D}\big((x+\alpha u^*)\hat{\bullet}(y+\beta u^*)\big)=\hat{D}\big(x\bullet y+ B(y,D(x))u^*\big)=D(x,y)-B(x_0,x\bullet y)u^*=D(x,y).
    \end{align*}
    On the other hand,
    \begin{align*}
        &\hat{D}(x+\alpha u^*)\hat{\bullet}(y+\beta u^*)=D(x)\bullet y+B(y,D^2(x))u^*
    ,\\
        &(x+\alpha u^*)\hat{\bullet}D(y+\beta u^*)=x\bullet D(y)+B(D(y),D(x))u^*.
    \end{align*}
    Then $\hat{D}$ is an anti superderivation if and only if
    $$B\big(y,D^2(x)\big)-B\big(D(y),D(x)\big)=0.$$
    Since $D$ of degree $1$ , supersymmetric with respect to $B$, and $D^2(x)=0.$, we have 
    \begin{align*}
        &B\big(y,D^2(x)\big)-B\big(D(y),D(x)\big)\\&=B\big(y,D^2(x)\big)+B\big(y,D^2(x)\big)=2B\big(y,D^2(x)\big)=0.
    \end{align*}
    Then $\hat{D}$ is an anti-superderivation. Now we have
    \begin{align*}
        &\hat{x}_0\hat{\bullet}(x+\alpha u^*)=(x_0+\lambda u^*)\hat{\bullet}(x+\alpha u^*)=x_0\bullet x+B(x,D(x_0))u^*=x_0\bullet x=0,\\
        &(x+\alpha u^*)\hat{\bullet}\hat{x}_0=(x+\alpha u^*)\hat{\bullet}(x_0+\lambda u^*)=x\bullet x_0+B(x_0,D(x))u^*=x\bullet x_0=0.
    \end{align*}
    
    Finaly $\hat{D}^2(x+\alpha u^*)=\hat{D}\big(D(x)-B(x_0,x)u^*\big)=D^2(x)-B\big(x_0,D(x)\big)=0.$ And $\hat{D}(x_0+\lambda u^*)=D(x_0)-B(x_0,x_0)u^*=0.$
    Then $(\hat{D},\hat{x_0})$ is an admissible pair.
\end{proof}
\begin{thm}
    Let $(\mathcal{J},\bullet,B)$ be a pseudo-euclidean mock-Lie superalgebra, $\mathcal{U}=\mathcal{U}_{\bar{1}}=\mathbb{K}u$ be the one-dimensional odd mock-Lie superalbegra, and $\mathcal{U}^*$ its dual space. Let $(D,x_0)$ be an admissible pair of $\mathcal{J}$ such that $D$ is super symmetric with respect to $B$, and $B(x_0,x_0)=0.$ Then $\tilde{\mathcal{J}}=\mathcal{U}\oplus\mathcal{J}\oplus\mathcal{U}^*$ with the product 
    \begin{align*}
        &u\tilde{\bullet}u=x_0+\lambda u^*,\quad \lambda\in\mathbb{K}\\
        &u\tilde{\bullet}x=D(x)-B(x_0,x)u^*,\quad\forall x\in\mathcal{J}\\
        &x\tilde{\bullet}y=x\bullet y+B(y,D(x))u^*,\quad\forall x,y\in\mathcal{J}\\
        &u^*\tilde{\bullet}X=X\tilde{\bullet}u^*=0, \quad\forall X\in\tilde{\mathcal{J}},
    \end{align*}
    is a mock-Lie superalgebra generalized double extension of $\mathcal{J}$ by the one-dimensional odd mock-Lie superalbegra $\mathcal{U}$ by means of $(D,x_0)$. Moreover the bilinear form $\tilde{B}$ defined by:\\
    $\tilde{B}(u,u^*)=1,\tilde{B}(u^*,u^*)=\tilde{B}(u,u)=0,\tilde{B}(u,x)=\tilde{B}(u^*,x)=0,\quad\forall x\in\mathcal{J},$ and $\tilde{B}_{|\mathcal{J}\times\mathcal{J}}=B,$ is an invariant scalar product on $\tilde{\mathcal{J}}.$
\end{thm}
\begin{proof}
    By Lemma \ref{central-by-one} and Lemma \ref{ext-pair}, $\hat{\mathcal{J}}$ is a mock-Lie superalgebra central extension of $\mathcal{J}$ by $\mathcal{U}^*$ and $(\hat{D},\hat{x}_0)$ is an admissible pair, by Lemma \ref{gen-semi-prod-by-one-dim}, $\tilde{\mathcal{J}}$ with the product  difined above is the generalized double extension of $\hat{\mathcal{J}}$ by $\mathcal{U}$. Clearly $\tilde{B}$ is even, supersymmetric, invariant, and non-degenerate.  
\end{proof}
\begin{rmk}
    If $Ann(\mathcal{J})\cap{\mathcal{J}}_{\bar{0}}={0},$ (i.e. $x_0=0).$  Then $\tilde{\mathcal{J}}$ defined above is a double extension of $\mathcal{J}$ by $\mathcal{U}.$ With the product 
     \begin{align*}
        &u\tilde{\bullet}u=0,\\
        &u\tilde{\bullet}x=D(x),\quad\forall x\in\mathcal{J}\\
        &x\tilde{\bullet}y=x\bullet y+B(y,D(x)),\quad\forall x,y\in\mathcal{J}\\
        &u^*\tilde{\bullet}X=X\tilde{\bullet}u^*=0, \quad\forall X\in\tilde{\mathcal{J}}.
    \end{align*}
\end{rmk}

\begin{defi}
    Let $(\mathcal{J}_1,B_1)$ and $(\mathcal{J}_2,B_2)$ be tow pseudo-euclidean mock-Lie superalgebras. An isometry from $(\mathcal{J}_1,B_1)$ to $(\mathcal{J}_2,B_2)$ is an isomorphism $\Psi:(\mathcal{J}_1,B_1)\to (\mathcal{J}_2,B_2)$ of mock-Lie superalgebras such that $B_2(\Psi(x),\Psi(y))=B_1(x,y),\quad\forall x,y\in\mathcal{J}_1.$
\end{defi}
Now, let us investigate the notion of isometry betwean tow generalized double extensions of pseudo-euclidean mock-Lie superakgebra $(\mathcal{J},\bullet,B).$ 
Let $(\mathcal{J},\bullet,B)$ be a mock-lie superalgebra, $({\tilde{\mathcal{J}}}_1,{\tilde{\bullet}}_1,{\tilde{B}}_1)$ (resp.$({\tilde{\mathcal{J}}}_2,{\tilde{\bullet}}_2,{\tilde{B}}_2)$ the generalized double extension of $({{\mathcal{J}}},{{\bullet}},{{B}})$ by $\mathcal{U}_1$ (resp.$\mathcal{U}_2$) by means of $(D_1,x_1)$ (resp.$(D_2,x_2)).$ Let $\Psi$ be an isometry from $({\tilde{\mathcal{J}}}_1,{\tilde{B}}_1)$ to $({\tilde{\mathcal{J}}}_2,{\tilde{B}}_2).$ Assume that $\Psi(\mathcal{J}\oplus{\mathcal{U}_1}^*)=\mathcal{J}\oplus{\mathcal{U}_2}^*.$ Then 
\begin{align*}
    &\Psi(u_1)=\gamma u_2+z_0+\mu {u_2}^*,\quad \gamma ,\mu\in \mathbb{K}, z_0\in \mathcal{J}\\
    &\Psi({u_1}^*)=\alpha{u_2}^*+u_0,\quad u_0\in\mathcal{J},\alpha\in\mathbb{K}\\
    &\Psi(x)=s(x)+t(x){u_2}^*.\quad\forall x\in\mathcal{J},
\end{align*}
where $s:\mathcal{J}\to\mathcal{J}$ is an even linear map   and $t$ is a linear form of $\mathcal{J}.$ 
Since ${\tilde{B}}_1$ is non degenerate then there exists $c\in\mathcal{J}$ such that $t(x)={\tilde{B}}_1(c,x).$ Let $x\in\mathcal{J},$ we have $0={\tilde{B}}_1({u_1}^*,x)={\tilde{B}}_2\big(\Psi({u_1}^*),\Psi(x)\big),$ then ${\tilde{B}}_2\big(u_0,s(x)\big)=0.$ And $0={\tilde{B}}_1(u_1,x)={\tilde{B}}_2\big(\Psi(u_1),\Psi(x)\big),$ then ${\tilde{B}}_2\big(z_0,s(x)\big)=-\gamma {\tilde{B}}_1(c,x).$ Moreover $0={\tilde{B}}_1({u_1}^*,{u_1}^*)={\tilde{B}}_2\big(\Psi({u_1}^*),\Psi({u_1}^*)\big)={\tilde{B}}_2(u_0,u_0).$

Now, let $x,y\in\mathcal{J}.$ ${\tilde{B}}_2\big(\Psi(x),\Psi(y)\big)={\tilde{B}}_2\big(s(x),s(y)\big).$ Then $s$ is bijective. Since ${\tilde{B}}_2$ is non degenerate then ${\tilde{B}}_2\big(u_0,s(x)\big)=0$ implies that $u_0=0,$ and then $\Psi({u_1}^*)=\alpha {u_2}^*.$ Further we have $1={\tilde{B}}_1(u_1,{u_1}^*)={\tilde{B}}_2\big(\Psi(u_1),\Psi({u_1}^*)\big)={\tilde{B}}_2(\gamma u_2+z_0+\mu {u_2}^*,\alpha{u_2}^*)=\alpha\gamma.$ Thus ${\tilde{B}}_2\big(z_0,s(x)\big)=-\gamma {\tilde{B}}_1(c,x).$ So ${\tilde{B}}_2\big(s^{-1}(z_0),x\big)={\tilde{B}}_1(-\gamma c,x).$ Then $c=-\alpha s^{-1}(z_0).$ Furthermore ${\tilde{B}}_1(u_1,u_1)={\tilde{B}}_2\big(\Psi(u_1),\Psi(u_1)\big)=2\gamma \mu+{\tilde{B}}_2(z_0,z_0)=0.$ Therfore ${\tilde{B}}_2(z_0,z_0)=2\gamma\mu.$ Then $\mu={\tilde{B}}_2(-\frac{\alpha}{2}z_0,z_0).$ Then $\Psi$ defined by 
\begin{align*}
    &\Psi(u_1)=\frac{1}{\alpha}u_2+z_0+\frac{-\alpha}{2}{B}(z_0,z_0){u_2}^*,\\
    &\Psi({u_1}^*)=\alpha{u_2}^*,\\
    &\Psi(x)=s(x)-\alpha B\big(z_0,s(x)\big){u_2}^*,\quad\forall x\in\mathcal{J}.
\end{align*}
Let $({\tilde{\mathcal{J}}}_i,{\tilde{\bullet}}_i,{\tilde{B}}_i)$, for $i=1,2,$ be a generalized double extension of $\mathcal{J}$ by $\mathcal{U}_i=\mathbb{K}u_i$. Then 
\begin{align*}
        &u_i\tilde{\bullet}_iu_i=x_i+\lambda_i u_i^*,\quad \lambda_i\in\mathbb{K}\\
        &u_i\tilde{\bullet}_ix=-x\tilde{\bullet}_iu_i=D_i(x)-B(x_i,x)u_i^*,\quad\forall x\in\mathcal{J}\\
        &x\tilde{\bullet}_iy=x\bullet y+B\big(y,D_i(x)\big)u_i^*,\quad\forall x,y\in\mathcal{J}\\
        &u_i^*\tilde{\bullet}_iX=X\tilde{\bullet}_iu_i^*=0, \quad\forall X\in\tilde{\mathcal{J}}.
    \end{align*}
With the above notation we have the following 
\begin{thm}
Let $({\tilde{\mathcal{J}}}_1,{\tilde{\bullet}}_1,{\tilde{B}}_1),({\tilde{\mathcal{J}}}_2,{\tilde{\bullet}}_2,{\tilde{B}}_2)$ be two generalized double extensions of $\mathcal{J}$. We say that $({\tilde{\mathcal{J}}}_1,{\tilde{\bullet}}_1,{\tilde{B}}_1)$ and $({\tilde{\mathcal{J}}}_2,{\tilde{\bullet}}_2{\tilde{B}}_2)$ is isometric if and only if there exist an isometry $s$ in $(\mathcal{J},{\bullet},{B}),$ $z_0\in Ann(\mathcal{J})\cap\mathcal{J}_{\bar{0}},$ and $\alpha\in\mathbb{K}\setminus\{0\}$ such that:
\begin{align*}
&\alpha^3\lambda_1-\lambda_2=B\big(z_0,\alpha^3s(x_1)\big),\\
    &s(x_1)=\frac{1}{\alpha^2}x_2,\\
    &s\circ D_1\circ s^{-1}-\frac{1}{\alpha}D_2=0.
\end{align*}
\end{thm}
\begin{proof}
Using the fact that $\Psi$ is a homomorphism. We have $$\Psi(u_1\tilde{\bullet}_1u_1)=\Psi(x_1+\lambda_1{u_1}^*)=\Psi(x_1)+\lambda_1\Psi({u_1}^*)=s(x_1)-\alpha B(z_0,s(x_1)){u_2}^*+\alpha\lambda_1{u_2}^*.$$ On the other hand, $$\Psi(u_1)\tilde{\bullet}_2\Psi(u_1)=\frac{1}{\alpha^2}u_2\tilde{\bullet}_2u_2+z_0\tilde{\bullet}_2z_0=\frac{1}{\alpha^2}x_2+z_0\bullet z_0+\frac{1}{\alpha^2}\lambda_2{u_2}^*+B(z_0,D_2(z_0)){u_2}^*.$$  Then we have 
\begin{equation}\label{iso5.1}
    s(x_1)=\frac{1}{\alpha^2}x_2+z_0\bullet z_0.
\end{equation}
And $\alpha^3\lambda_1-\lambda_2=\alpha^3B(z_0,s(x_1))+\alpha^2B(z_0,D_2(z_0)).$ Since $B$ is non degenerate, then
\begin{equation}
    \alpha^3\lambda_1-\lambda_2=\alpha^3B(z_0,s(x_1)).
\end{equation}
Now, let $x\in\mathcal{J}.$ Then $$\Psi(u_1\tilde{\bullet}x)=\Psi(D_1(x)-B(x_1,x){u_1}^*)=s(D_1(x))-\alpha B(z_0,s(D_1(x))){u_2}^*-\alpha B(x_1,x){u_2}^*.$$ On the other hand, $$\Psi(u_1)\tilde{\bullet}_2\Psi(x)=\frac{1}{\alpha}D_2(s(x))+z_0\bullet s(x)-\frac{1}{\alpha}B(x_2,s(x)){u_2}^*+B(s(x),D_2(z_0)){u_2}^*.$$ Then 
\begin{equation*}
    s(D_1(x))-\frac{1}{\alpha}D_2(s(x))=z_0\bullet s(x).
\end{equation*}
Replacing $x$ by $s^{-1}(x),$ yields  
\begin{equation}\label{iso5.3}
    s\circ D_1\circ s^{-1}(x)-\frac{1}{\alpha}D_2(x)=z_0\bullet x.
\end{equation}
Moreover, $-\alpha B(z_0,s(D_1(x))-\alpha B(x_1,x)+\frac{1}{\alpha}B(x_2,s(x))-B(s(x),D_2(z_0))=0.$ Using the properties of $D_1$ and $s.$ We have $B(D_1(s^{-1}(z_0))+x_1-\frac{1}{\alpha^2}s^{-1}(x_2)-\frac{1}{\alpha}s^{-1}(D_2(z_0)),x)=0.$ Since $B$ non degenerate, then $D_1(s^{-1}(z_0)+x_1-\frac{1}{\alpha^2}s^{-1}(x_2)-\frac{1}{\alpha}s^{-1}(D_2(z_0))=0.$ Applying $s$ we have $s(D_1(s^{-1}(z_0))+s(x_1)-\frac{1}{\alpha^2}x_2-\frac{1}{\alpha}D_2(z_0)=0.$ Finaly using \eqref{iso5.1}, we have
\begin{equation}\label{iso5.4}
    s\circ D_1\circ s^{-1}(z_0)-\frac{1}{\alpha}D_2(z_0)=-z_0\bullet z_0.
\end{equation}
Simulary $\Psi(x\tilde{\bullet}_1u_1)=\Psi(x)\tilde{\bullet}_2\Psi(u_1)$ equivalent to  
\begin{equation}\label{iso5.5}
     s\circ D_1\circ s^{-1}(x)-\frac{1}{\alpha}D_2(x)=-z_0\bullet x.
\end{equation}
And \eqref{iso5.4} is satisfied. By \eqref{iso5.3} and \eqref{iso5.5}, clearly $z_0\bullet x=x\bullet z_0=0.$ And $s\circ D_1\circ s^{-1}(x)-\frac{1}{\alpha}D_2(x)=0.$ Then, according to  \eqref{iso5.1}, we have $s(x_1)=\frac{1}{\alpha^2}x_2.$ Now let $x,y\in\mathcal{J}.$ Thus we have, $$\Psi(x\tilde{\bullet}_1y)=s(x\bullet y)-\alpha B(z_0,s(x\bullet y)){u_2}^*+\alpha B(y,D_1(x)) {u_2}^*.$$ On the other hand $\Psi(x)\tilde{\bullet}_2\Psi(y)=s(x)\bullet s(y)+B(s(y),D_2(s(x)){u_2}^*.$ Then $s(x\bullet y)=s(x)\bullet s(y),$ and \eqref{iso5.3} are satisfied. 
\end{proof}

\begin{thm}\label{doub-ex}
Let $(\tilde{\mathcal{J}},\tilde{\bullet},\tilde{B)}$ be a pseudo-euclidean mock-Lie superalgebra such that $dim(\tilde{\mathcal{J}}_{\bar{1}})>1.$ Then $\tilde{\mathcal{J}}$ is a generalized double extension of a pseudo-euclidean mock-Lie superalgebra $(\mathcal{J},\bullet,B)$ by the one-dimensional odd mock-Lie superalgebra such that $dim(\mathcal{J}_{\bar{1}})=n-2$.
\end{thm}
\begin{proof}
    By Corollary \ref{central-element}, there exists $u^*\in Ann(\tilde{\mathcal{J}})\cap \tilde{\mathcal{J}}_{\bar{1}}\neq \{0\}.$ Let $\mathcal{U^*}=\mathbb{K}u^*$ and $({\mathcal{U}^*})^{\perp}$ its orthogonal space with respect to $\tilde{B}$. Then $\mathcal{U}^*\subseteq(\mathcal{U}^*)^\perp$. Since $\tilde{B}$ is even and non degenerate there exists $u\in\tilde{\mathcal{J}}_{\bar{1}}$ such that $\tilde{B}(u^*,u^*)=0$ and $\tilde{B}(u^*,u)=1.$ Let $\mathcal{U}=\mathbb{K}u$ and $\mathcal{J}=(\mathcal{U}\oplus\mathcal{U}^*)^{\perp}$, then $\tilde{\mathcal{J}}=\mathcal{U}\oplus\mathcal{J}\oplus\mathcal{U}^*.$ Since $\tilde{B}$ is invariant then $\tilde{\mathcal{J}}^2\subseteq(\mathcal{U}^*)^\perp.$ It follows that $(\mathcal{U^*})^{\perp}=\mathcal{U^*}\oplus\mathcal{J}$ is an ideal of $\tilde{\mathcal{J}}.$Therefore 
    \begin{align*}
        &u\tilde{\bullet}u=\alpha u+x_0+\beta u^*,\quad\forall\alpha,\beta\in\mathbb{K}\\
        &u\tilde{\bullet}x=D(x)+\eta(x)u^*,\quad\forall x\in\mathcal{J}\\
        &x\tilde{\bullet}y=x\bullet y+\varphi(x,y)u^*,\quad \forall x,y \in \mathcal{J}\\
        &u^*\tilde{\bullet}X=X\bullet_du^*=0, \quad\forall X\in\tilde{\mathcal{J}}.
    \end{align*}
    Where $''\bullet''$ is the product  of $\mathcal{J}$, $D$ is an endomorphism of $\mathcal{J}$ of degree $1$, $\eta$ is a linear form, and $\varphi$ is a bilinear form of $\mathcal{J}$.\\
    Since $\tilde{B}$ is even and invariant then we have 
    \begin{align*}
        &\alpha=0,\\
        &\varphi(x,y)=\tilde{B}(x\tilde{\bullet}y,u)=\tilde{B}(x,y\tilde{\bullet}u)=-\tilde{B}(x,D(y))=\tilde{B}(y,D(x)),\\
        &\eta(x)=-\tilde{B}(x\tilde{\bullet}u,u)=\tilde{B}(x,u\tilde{\bullet}u)=-\tilde{B}(x_0,x),\\
        &B(x\bullet y,z)=B(x,y\bullet z), \forall\quad x,y,z\in\mathcal{J}.
    \end{align*}
    Further $\tilde{B}(u,x\tilde{\bullet}y)=\tilde{B}(u\tilde{\bullet}x,y)=\tilde{B}(D(x),y),$ on the other hand $\tilde{B}(u,x\tilde{\bullet}y)=(-1)^{|x||y|}\tilde{B}(u,y\tilde{\bullet}x)=(-1)^{|x||y|}\tilde{B}(u\tilde{\bullet}y,x)=(-1)^{|x||y|}\tilde{B}(D(y),x)=(-1)^{|D||x|}\tilde{B}(x,D(y)).$ Then $D$ is super symmetric with respect to $B$, where $B=\tilde{B}_{|\mathcal{J}\times\mathcal{J}}$. Moreover $\tilde{B}(x_0\tilde{\bullet}u,u)=-\tilde{B}(x_0,x_0)=\tilde{B}(x_0,u\tilde{\bullet}u)=\tilde{B}(x_0,x_0),$ then $\tilde{B}(x_0,x_0)=0.$ Using the super-Jacobi identity of $\tilde{\mathcal{J}}$ with the three elements $x,u$ and $u$ we prove that $x_0\bullet x=0,\quad\forall x\in\mathcal{J}$ and then $D(x_0)=0.$ With the three elements $x,y,u$ we prove that $D$ is an anti superderivation and $D^2=0$. Finaly it is easy to prove the supercommutativity and the super-Jacobi identity of ${\mathcal{J}}.$ Then $\tilde{\mathcal{J}}=\mathcal{U}\oplus\mathcal{J}\oplus\mathcal{U}^*$ is a generalized double extension of $\mathcal{J}$ by the one-dimensional odd mock-Lie superalgebra $\mathcal{U}.$
\end{proof}
Let $(\tilde{\mathcal{J}},\tilde{\bullet},\tilde{B)}$ be a pseudo-euclidean mock-Lie superalgebra such that $dim(\tilde{\mathcal{J}}_{\bar{1}})=n>1.$  By induction on $n$, if $n=2$ then $\tilde{\mathcal{J}}$ is a double extension of superalgebra $\{0\}$ by $\mathcal{U}$. Suppose that the result is true for every $m<n$. Then, by Theorem \ref{doub-ex}, $\tilde{\mathcal{J}}$ is a generalized double extension or double extension of $\mathcal{J}$ by $\mathcal{U}$, and $dim(\mathcal{J})=n-2$. Then applying the induction hypothesis to $\mathcal{J}$.
\begin{cor}
    Let $(\tilde{\mathcal{J}},\tilde{B})$ be a pseudo-euclidean mock-Lie superalgebra, such that $ dim(\tilde{\mathcal{J}}_{\bar{1}})=n>1.$ Then $(\tilde{\mathcal{J}},\tilde{B})$ is obtained from a pseudo-euclidean mock-Lie superalgebra by a finite number of generalized double extensions and / or double extensions by the one-dimensional odd mock-Lie superalgebra.
    \end{cor}
\section*{Open Questions}

In \cite{Braiek}, the authors present the concept of anti-Leibniz algebras, described as a "non-commutative version" of mock-Lie algebras. They explore the use of averaging operators and, more broadly, embedding tensors to construct new algebraic structures. The following research questions merit careful consideration.
\begin{itemize}
    \item We will introduce the concept of anti-Leibniz superalgebras, which can be viewed as a "non-commutative analogue" of mock-Lie superalgebras. A classification of these algebras in low dimensions will be provided, highlighting their structural properties and differences from their commutative counterparts. We will then investigate averaging operators and, more generally, embedding tensors as tools for constructing new graded algebraic structures, including various extensions and generalizations within the anti-Leibniz framework.
    \item In \cite{Said}, the authors examine Leibniz algebras equipped with symmetric, nondegenerate, and associative bilinear forms, referred to as quadratic Leibniz algebras. They demonstrate that Leibniz algebras possessing such structures remain within the class of Leibniz algebras and provide several concrete examples. Additionally, they reduce the study of quadratic Leibniz algebras to that of quadratic Lie algebras by introducing specific extensions of Leibniz algebras. The 
$\mathbb Z_2$-graded counterpart of this theory, known as quadratic Leibniz superalgebras, is investigated in \cite{Fahmi}. Building upon earlier works, we introduce the notion of quadratic anti-Leibniz (super)algebras, defined as anti-Leibniz (super)algebras equipped with an (even) (super)symmetric, non-degenerate, and associative bilinear form. We provide characterizations of these structures and propose an inductive framework for describing all quadratic anti-Leibniz (super)algebras. This approach allows us to reduce the study of such algebras to that of quadratic mock-Lie (super)algebras. Moreover, the description yields an explicit algorithm for constructing quadratic anti-Leibniz (super)algebras.
\end{itemize}

\end{document}